\documentclass[sn-mathphys,Numbered]{sn-jnl}

\usepackage{amsmath,amssymb,amsfonts}%
\usepackage{amsthm}%
\usepackage{mathrsfs}%
\usepackage[title]{appendix}%
\usepackage{xcolor}%
\usepackage{textcomp}%
\usepackage{manyfoot}%
\usepackage{listings}%

\usepackage{graphicx}
\usepackage{subfigure}
\usepackage{color}
\usepackage{bm}
\usepackage{algorithm}
\usepackage{algorithmicx}
\usepackage{algpseudocode}
\usepackage{multirow}
\usepackage{booktabs}
\usepackage{cases}
\usepackage{epstopdf}
\usepackage{enumerate}

\usepackage{pdfsync}

\theoremstyle{thmstyleone}%
\newtheorem{theorem}{Theorem}
\newtheorem{proposition}[theorem]{Proposition}%

\theoremstyle{thmstyletwo}%
\newtheorem{example}{Example}%
\newtheorem{remark}{Remark}%

\theoremstyle{thmstylethree}%
\newtheorem{lemma}[theorem]{Lemma}

\DeclareMathOperator{\sech}{sech}

\newcommand{\IR}{{\mathbb{R}}}

\newcommand{\IZ}{{\mathbb{Z}}}

\newcommand{\mO}{{\mathcal{O}}}
\newcommand{\mH}{{\mathcal{H}}}
\newcommand{\mL}{{\mathcal{L}}}

\newcommand{\bi}{{\bf i}}

\newcommand{\bv}{\bm{v}}

\newcommand{\be}{\bm{e}}

\newcommand{\bx}{\bm{x}}
\newcommand{\by}{\bm{y}}

\newcommand{\mF}{\mathcal{F}}

\newcommand{\mC}{\mathcal{C}}
\newcommand{\mI}{\mathcal{I}}

\newcommand{\bA}{\bm{A}}
\newcommand{\bB}{\bm{B}}
\newcommand{\bU}{\bm{U}}
\newcommand{\bI}{\bm{I}}
\newcommand{\bH}{\bm{H}}
\newcommand{\bJ}{\bm{J}}
\newcommand{\bC}{\bm{C}}
\newcommand{\bS}{\bm{S}}
\newcommand{\bbe}{\bm{e}}
\newcommand{\bc}{\bm{c}}
\newcommand{\bs}{\bm{s}}
\newcommand{\bF}{\bm{F}}
\newcommand{\bQ}{\bm{Q}}
\newcommand{\blam}{\bm{\lambda}}
\newcommand{\contran}{\mathsf{H}}
\newcommand{\diag}{{\rm diag}}
\newcommand{\fft}{{\rm fft}}
\newcommand{\ifft}{{\rm ifft}}
\newcommand{\diff}{{\rm d}}
\newcommand{\MtoV}{{\rm vec}}

\begin{document}
\title[Splitting ADI scheme for fractional Laplacian wave equations]{Splitting ADI scheme for fractional Laplacian wave equations}

\author[1]{\fnm{Tao} \sur{Sun}}\email{airfyst@163.com}

\author*[1]{\fnm{Hai-Wei} \sur{Sun}}\email{hsun@um.edu.mo}
\equalcont{These authors contributed equally to this work.}

\affil*[1]{\orgdiv{Department of Mathematics}, \orgname{University of Macau}, \orgaddress{ \city{Macao}, \postcode{999078}, \country{China}}}

\abstract{In this paper, we investigate the numerical solution of the two-dimensional fractional Laplacian wave equations.
After splitting out the Riesz fractional derivatives from the fractional Laplacian, we treat the Riesz fractional derivatives with an implicit scheme while solving the rest part explicitly. Thanks to the tensor structure of the Riesz fractional derivatives, a splitting alternative direction implicit (S-ADI) scheme is proposed by incorporating an ADI remainder. 
Then the Gohberg-Semencul formula, combined with fast Fourier transform, is proposed to solve the derived Toeplitz linear systems at each time integration. 
Theoretically, we demonstrate that the S-ADI scheme is unconditionally stable and possesses second-order accuracy. Finally, numerical experiments are performed to demonstrate the accuracy and efficiency of the S-ADI scheme.}


\maketitle

\section{Introduction}\label{section:introduction}

In this paper, we consider the numerical solution for the following initial boundary value problem of the fractional Laplacian wave equation:
\begin{numcases}{}
\partial_t^2 u(x,y,t) = -\kappa(-\Delta)^{\frac{\alpha}{2}} u(x,y,t) + g(u(x,y,t)),\quad (x,y)\in \Omega,~t\in(0,T], \label{eq:wave_equation}\\
u(x,y,t) = 0, \hspace{5.26cm}\quad (x,y)\in \Omega^c,~t\in[0,T],\\
u(x,y,0) = \varphi_1(x,y), \hspace{4.21cm}\quad (x,y)\in\Omega,\\
\partial_{t} u(x,y,0) = \varphi_2(x,y), \hspace{3.91cm}\quad (x,y)\in\Omega,\label{eq:initial condition 2}
\end{numcases}
where $\alpha\in(1,2)$ is the order of the fractional Laplacian, $T>0$ is the length of the time interval, $\Omega=(a,b)^2$ is a square domain in $\IR^2$, $\Omega^c=\IR^2\setminus\Omega$ and $\varphi_1,\varphi_2$ are two given functions on $\Omega$.
For a function $v:\IR^d\mapsto \IR$, the fractional Laplacian is defined by \cite{Hao_Zhang_2021}
\[-(-\Delta)^{\frac{\alpha}{2}} v(\bx):=-\frac{2^\alpha \Gamma(\frac{\alpha+d}{2})}{\pi^{\frac{d}{2}}|\Gamma(-\frac{\alpha}{2})|}\text{p.v.}\int_{\IR^d} \frac{v(\bx)-v(\by)}{ |\bx-\by|^{d+\alpha} } \diff \by,\quad \bx\in \IR^d,\]
where `p.v.' means that this integral takes its Cauchy principal value.
When $d=1$, the fractional Laplacian is equivalent to the Riesz fractional derivative \cite{Celik_Duman_2012}
\[\partial_{x}^{\alpha}v(x)= -\frac{\partial_{x}^2\int\nolimits_{\IR}|x-\xi|^{1-\alpha}v(\xi)\diff\xi}{2\cos(\alpha\pi/2)\Gamma(2-\alpha)},\quad x\in\IR.\]

The fractional Laplacian wave equation \eqref{eq:wave_equation} with $\alpha=2$ is the classical integer order wave equation, which has wide-ranging applications in quantum field theory \cite{Peskin_1995}, particle physics \cite{Griffiths_1987}, and superconductor modelling \cite{Scott_2003}. In recent years, fractional calculus has garnered significant attention from researchers due to its successful applications in anomalous dispersion, random walk, and control systems \cite{Sun_Zhang_2018,Metzler_Klafter_2000,Sun_Zhang_2014}. As a result, the space fractional wave equation, which is a generalization of the integer order wave equation, have been extensively studied.
For the one-dimensional (1D) fractional Laplacian wave equation, also known as the 1D Riesz fractional wave equation, numerous related works exist \cite{Xing_Wen_2018,Macias_Hendy_2018,Macias_Hendy_2018_2,Fu_Zhao_2022}. For the two-dimensional (2D) case, Wang and Shi \cite{Wang_Shi_2021} proposed an energy-conserving exponential scalar auxiliary variable spectral scheme for the fractional Laplacian wave equation on an unbounded domain. Hu et al. \cite{Hu_Cai_2021} employed a dissipation-preserving Crank-Nicolson pseudo-spectral method to solve the fractional Laplacian sine-Gordon equation with damping. Guo et al. \cite{Guo_Yan_2022} combined the Crank-Nicolson scheme with exponential scalar auxiliary variable technique in the time direction and employed spectral-Galerkin method in the space direction to solve the coupled fractional Laplacian Klein-Gordon equation. 
However, the error estimations in those works are incomplete.
Recently, based on the progress made in \cite{Hao_Zhang_2021} about the finite difference approximation of the multi-dimensional fractional Laplacian operator, Hu et al. \cite{Hu_Cai_2021_2} proposed a dissipation-preserving difference scheme for the damped fractional Laplacian wave equation and established the unconditional stability and convergence of the proposed scheme.

It is important to note that solving linear systems generated by multi-dimensional space fractional derivatives is a computationally expensive task, as those systems typically have large and dense coefficient matrices. 
To address this issue, the alternative direction implicit (ADI) technique, proposed by Peaceman and Rachford \cite{Peaceman_Rachford_1955}, has been widely used. This technique transforms the multi-dimensional time-dependent partial differential equation into a series of one-dimensional problems at each time integration, effectively reducing the computational cost of solving the resulting linear systems. The ADI technique is particularly useful when solving Riesz fractional differential equations, and has been successfully applied in various studies \cite{Meerschaert_Scheffler_2006,Zeng_Liu_2014,Lin_Ng_2019}. However, it cannot be directly applied to fractional Laplacian differential equations due to their {\it non-tensorial structure}.

To address this issue, we introduce the operator splitting technique \cite{Rosales_Seibold_2017,Seibold_Shirokoff_2019,Wang_Zhang_2020,Sun_Sun}, which can alter the coefficient matrix of the linear systems appeared in numerical scheme by properly splitting the linear operator. In this paper, we split out the discrete Riesz fractional derivative from the spatial discretized fractional Laplacian wave equation. Then we use a second-order implicit-explicit (IMEX) scheme to discretize the resulted second-order ordinary differential equation systems. Then we add a remainder term into the IMEX scheme to achieve the so-call S-ADI scheme. The  coefficient matrices of the linear systems appearing in non-ADI scheme, IMEX scheme and S-ADI scheme are presented in Figure \ref{fig:idea}, where $\bI,\bI_x,\bI_y$ are identity matrices, $\bA$ denotes the coefficient matrix of the discrete fractional Laplacian, and $\bA_x,\bA_y$ represent the coefficient matrices of the discrete Riesz fractional derivatives in $x$-direction and $y$-direction, respectively. The solutions of the linear systems generated by the S-ADI scheme are obtained by solving a series of 1D problems, which makes it feasible to use Gohberg-Semencul (GS) formula  \cite{Lee_PAng_2010,Pang_Sun_2011,Chen_Sun_2021} to solved the derived Toeplitz linear systems. It is the main advantage compared with the non-ADI scheme.

\begin{figure}\label{fig:idea}
\centering
\includegraphics[width=0.45\textwidth]{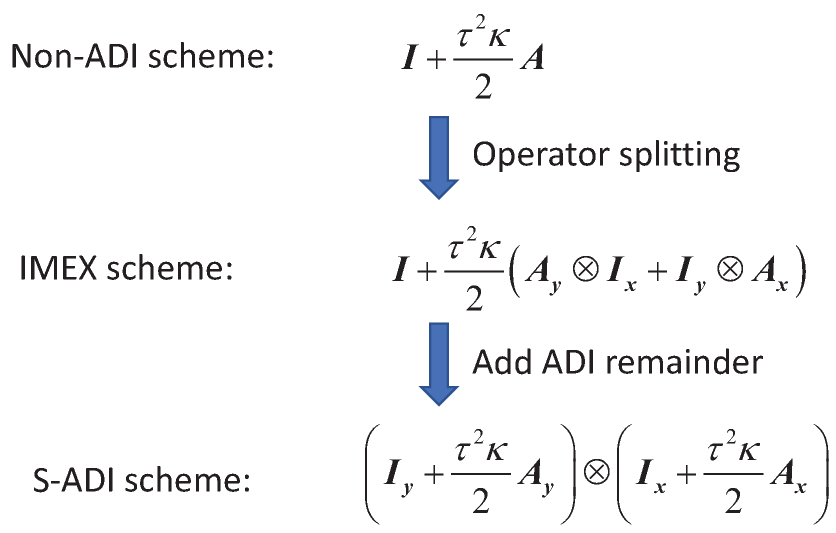}
\caption{The coefficient matrices of the linear systems from different numerical schemes. }
\label{grids_2d}
\end{figure}

In this study, we present a second-order accurate S-ADI scheme for solving the fractional Laplacian wave equation. The fast implementation based on GS formula is discussed in detail. Through rigorous analysis, we establish the unconditional stability of the S-ADI scheme. Moreover,  we demonstrate its unconditional convergence. Finally, we conduct numerical experiments to showcase the accuracy and efficiency of the proposed S-ADI scheme.

The rest of this paper is arranged as follows. In Section \ref{section:construction}, we construct the S-ADI scheme for the fractional Laplacian wave equation. 
Section \ref{section:implementation} discusses the fast implementation of the S-ADI scheme based on GS formula.  
In Section \ref{section:stability}, the unconditional stability of the S-ADI scheme is rigorously proven. After estimating the truncation error, we demonstrate the unconditional convergence of the S-ADI scheme in Section \ref{section:error}. 
In Section \ref{section:numerical}, the proposed scheme is applied to solve fractional Laplacian sine-Gordon and Klein-Gordon equations, and compared with the non-ADI scheme. 
Finally, conclusion of current work and future researches are outlined in Section \ref{section:conclusion}.

\section{Construction of S-ADI scheme}\label{section:construction}
Since the fractional Laplacian is a nonlocal operator, its discrete version cannot be represented as a tensor format in  multi-dimensional case. Hence, it is unlikely to construct ADI scheme directly for the differential equation with the fractional Laplacian. 
Nevertheless, we note that the operator splitting technique can alter the linear systems generated by numerical scheme without damaging the numerical stability.  
To achieve that, the splitter should possess strong stiffness that is not less than the original stiff operator. It is worth to mention that the fractional Laplacian operator has a twin brother,  the Riesz fractional derivative, which has the tensor structure. They are equivalent in 1D case, and different from each other in multi-dimensional case. 
Their discrete versions, however, possess similar stiffness. Therefore, after discretizing the fractional Laplacian, we split out the discrete Riesz fractional derivative from the linear part, and use IMEX scheme to discretize the time derivative. 
By adding an ADI remainder to the IMEX scheme, we get a fully discrete S-ADI scheme for the fractional Laplacian wave equation \eqref{eq:wave_equation}.

For given positive integers $N,M$, denote the time step size by $\tau=T/M$ and the space step size by $h=(b-a)/(N+1)$. Define the discrete grid points in $\Omega\times[0,T]$ as $(x_i,y_j,t_n)=(a+ih,a+jh,n\tau)$ for $1\le i,j\le N$ and $0\le n\le M$.
We first discretize the fractional Laplacian operator in \eqref{eq:wave_equation}. In \cite{Hao_Zhang_2021}, Hao, Zhang and Du proposed a difference approximation:
\begin{equation}\label{eq:center_disccritzation_fractional_Laplacian}
L_h^\alpha v(x,y) :=\frac{1}{h^\alpha}\sum_{i,j\in\IZ}a_{ij}^{(\alpha)}v(x+ih,y+jh)\approx (-\Delta)^{\frac{\alpha}{2}} v(x,y),
\end{equation}
where
\[a_{ij}^{(\alpha)}:=\frac{1}{4\pi^2}\int_{-\pi}^{\pi}\int_{-\pi}^{\pi} \left(4\sin^2\left(\frac{\eta}{2}\right) + 4\sin^2\left(\frac{\xi}{2}\right) \right)^{\frac{\alpha}{2}} \exp(-\bi (i\eta+j\xi))\diff \eta \diff \xi. \]
Apply \eqref{eq:center_disccritzation_fractional_Laplacian} to \eqref{eq:wave_equation} with $(x,y,t)=(x_i,y_j,t_n)$, we have
\begin{equation}\label{eq:space-discretized wave equation}
\partial_t^2 u(x_i,y_j,t_n) = -\kappa L_h^\alpha u(x_i,y_j,t_n) + g(u(x_i,y_j,t_n)) +R^{n,1}_{ij},\quad 0\le n\le M,
\end{equation}
where
\[R^{n,1}_{ij} =\kappa L_h^\alpha u(x_i,y_j,t_n) -\kappa(-\Delta)^{\frac{\alpha}{2}} u(x_i,y_j,t_n).\]
To split the linear operator $ L_h^\alpha$, we introduce the fractional central difference operators for Riesz fractional derivatives \cite{Celik_Duman_2012}:
\begin{equation}\label{eq:single_discrete_Riesz}
\delta_x^\alpha v(x,y):=\frac{1}{h^\alpha}\sum_{i\in\IZ} a^{(\alpha)}_{i} v(x+ih,y),\quad \delta_y^\alpha v(x,y):=\frac{1}{h^\alpha}\sum_{i\in\IZ} a^{(\alpha)}_{i} v(x,y+ih),
\end{equation}
where
\begin{equation}\label{eq:single coefficient discrete Riesz}
a^{(\alpha)}_{i}:=\frac{1}{2\pi}\int_{-\pi}^{\pi} \left(4\sin^2\left(\frac{\eta}{2}\right) \right)^{\frac{\alpha}{2}} \exp(-\bi i\eta)\diff \eta=\frac{(-1)^i\Gamma(\alpha+1)}{\Gamma(\alpha/2-i+1)\Gamma(\alpha/2+i+1)}.
\end{equation}
Split out the discrete 2D Riesz fractional difference operator $\delta_x^\alpha+\delta_y^{\alpha}$ from  $ L_h^\alpha$. 
Then the space discretized wave equation \eqref{eq:space-discretized wave equation} is rewritten as
\begin{equation}\label{eq:after_use stiff cutter}
\partial_t^2 U^n_{ij} =-\kappa(\delta_x^\alpha+\delta_y^{\alpha})U^n_{ij} -\kappa (L_h^\alpha-\delta_x^\alpha-\delta_y^{\alpha}) U^n_{ij} + g(U^n_{ij}) +R^{n,1}_{ij},
\end{equation}
where $U^n_{ij}:=u(x_i,y_j,t_n)$.

In Figure \ref{fig:points contribution of difference operators}, we illustrate the contribution points of operators $\delta^\alpha_x+\delta^\alpha_y$ and $L_h^\alpha$ on a 2D uniform mesh. 
We  observe that all grid points in the domain contribute to the considered discrete point for the fractional Laplacian operator, while only the points in the horizontal and vertical directions contribute for the Riesz fractional derivative operator. 
This fundamental difference between those two operators is crucial and enables the construction of an ADI type scheme by splitting out $\delta_x^\alpha+\delta_y^{\alpha}$.

\begin{figure}
    \centering

    \subfigure[$\delta^\alpha_x+\delta^\alpha_y$]{
    \begin{minipage}[c]{0.4\textwidth}
    \centering
    \includegraphics[width=1\textwidth]{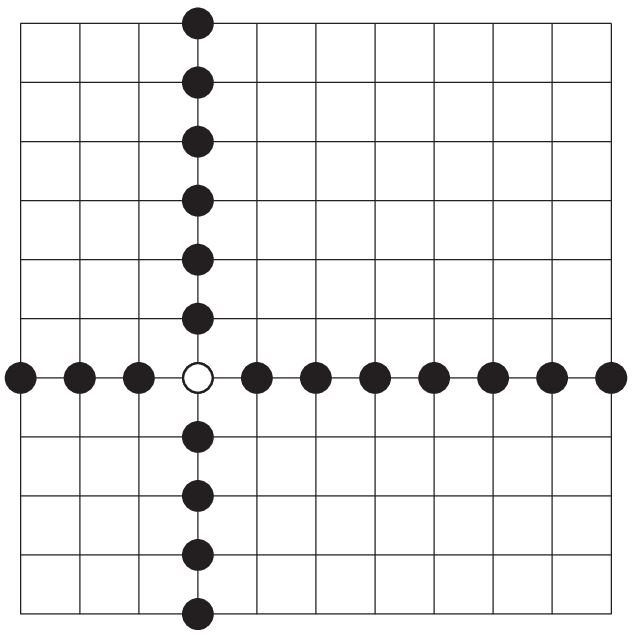}
    \end{minipage}
}
\!\!\!
   \subfigure[$L_h^\alpha$]{
    \begin{minipage}[c]{0.4\textwidth}
    \centering
    \includegraphics[width=1\textwidth]{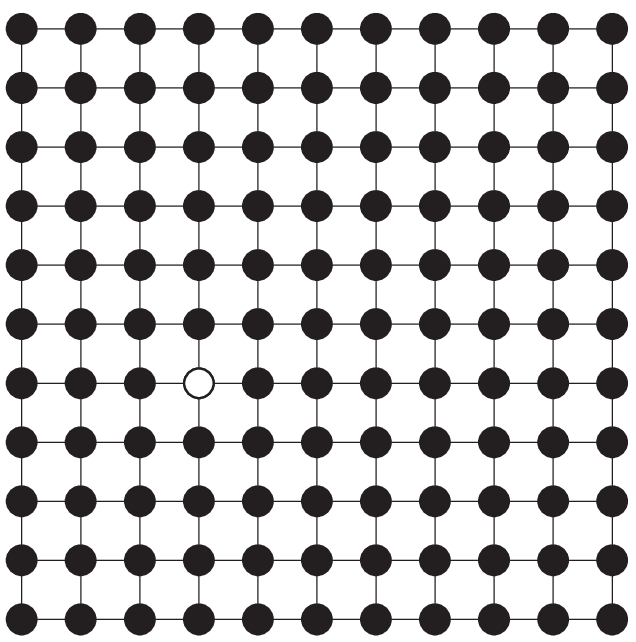}
    \end{minipage}
}
    \caption{The contribution points of the fractional difference operators: the solid point contributes to the hollow point.}
    \label{fig:points contribution of difference operators}
\end{figure}

Now consider two cases: $n\ge1$ and $n=0$.
For $n\ge1$, applying second-order central difference formula $\frac{U^{n+1}_{ij}-2U^{n}_{ij}+U^{n-1}_{ij}}{\tau^2}$ to approximate $\partial_t^2 U^n_{ij}$ and middle point formula $(\delta_x^\alpha+\delta_y^{\alpha})\frac{U^{n+1}_{ij}+U^{n-1}_{ij}}{2}$ to approximate $(\delta_x^\alpha+\delta_y^{\alpha})U^n_{ij}$ gives the discretization:
\begin{equation}\label{eq:before ADI}
\begin{split}
\frac{U^{n+1}_{ij}-2U^{n}_{ij}+U^{n-1}_{ij}}{\tau^2} 
=&-\kappa(\delta_x^\alpha+\delta_y^{\alpha})\frac{U^{n+1}_{ij}+U^{n-1}_{ij}}{2} -\kappa (L_h^\alpha-\delta_x^\alpha-\delta_y^{\alpha}) U^n_{ij} \\
& + g(U^n_{ij}) +R^{n,1}_{ij}+R^{n,2}_{ij},
\end{split}
\end{equation}
where
\[R^{n,2}_{ij}:=\frac{U^{n+1}_{ij}-2U^{n}_{ij}+U^{n-1}_{ij}}{\tau^2}-\partial_t^2 U^n_{ij}  +\kappa(\delta_x^\alpha+\delta_y^{\alpha})\frac{U^{n+1}_{ij}-2U^n_{ij}+U^{n-1}_{ij}}{2}.\]
To achieve the S-ADI scheme, we add a small perturbation term $R^{n,3}_{ij}:=\frac{\kappa^2\tau^2}{4}\delta_x^\alpha \delta_y^\alpha(U^{n+1}_{ij}-2U^{n}_{ij}+U^{n-1}_{ij})$ to \eqref{eq:before ADI}, which gives
\begin{equation}\label{eq:ADI with truncation error}
\begin{split}
\frac{U^{n+1}_{ij}-2U^{n}_{ij}+U^{n-1}_{ij}}{\tau^2} =&-\kappa(\delta_x^\alpha+\delta_y^{\alpha})\frac{U^{n+1}_{ij}\!+\!U^{n-1}_{ij}}{2} -\frac{\kappa^2\tau^2}{4}\delta_x^\alpha \delta_y^\alpha(U^{n+1}_{ij}\!-\!2U^{n}_{ij}\!+\!U^{n-1}_{ij})\\
 &-\kappa (L_h^\alpha-\delta_x^\alpha-\delta_y^{\alpha}) U^n_{ij} + g(U^n_{ij}) +R^{n,1}_{ij}+R^{n,2}_{ij}+R^{n,3}_{ij}.
\end{split}\end{equation}
By dropping the truncation error $R^{n,1}_{ij}+R^{n,2}_{ij}+R^{n,3}_{ij}$ and substituting the analytical solution $U^n_{ij}$ by the approximate solution $u^n_{ij}$, the S-ADI scheme is obtained  for the fractional Laplacian wave equation \eqref{eq:wave_equation}:
\begin{equation}\label{eq:ADI scheme}
\begin{split}
\frac{u^{n+1}_{ij}-2u^{n}_{ij}+u^{n-1}_{ij}}{\tau^2} =&-\kappa(\delta_x^\alpha+\delta_y^{\alpha})\frac{u^{n+1}_{ij}+u^{n-1}_{ij}}{2} -\frac{\kappa^2\tau^2}{4}\delta_x^\alpha \delta_y^\alpha(u^{n+1}_{ij}-2u^{n}_{ij}+u^{n-1}_{ij})\\
 &-\kappa (L_h^\alpha-\delta_x^\alpha-\delta_y^{\alpha}) u^n_{ij} + g(u^n_{ij}), \quad n\ge1.
\end{split}\end{equation}

Since \eqref{eq:ADI scheme} is a three-step scheme, we still require a two-step scheme to obtain the numerical solution at $t=\tau$.
For $n=0$, we use $(U^{1}_{ij}-U^{0}_{ij}-\tau \partial_t U^0_{ij})/(\tau^2/2)$ and $(\delta_x^\alpha+\delta_y^\alpha)U^1_{ij}$ to approximate $\partial_t^2 U^0_{ij}$ and $(\delta_x^\alpha+\delta_y^\alpha)U^0_{ij}$ in \eqref{eq:after_use stiff cutter}, respectively:
\begin{equation}\label{eq:before ADI initial}
\frac{U^{1}_{ij}-U^{0}_{ij}-\tau \partial_t U^0_{ij}}{\tau^2/2} =-\kappa(\delta_x^\alpha+\delta_y^{\alpha})U^{1}_{ij} -\kappa (L_h^\alpha-\delta_x^\alpha-\delta_y^{\alpha}) U^0_{ij} + g(U^0_{ij}) +R^{0,1}_{ij}+R^{0,2}_{ij},
\end{equation}
where
\[R^{0,2}_{ij}:=\frac{U^{1}_{ij}-U^{0}_{ij}-\tau \partial_t U^0_{ij}}{\tau^2/2}-\partial_t^2 U^0_{ij}  +\kappa(\delta_x^\alpha+\delta_y^{\alpha})(U^{1}_{ij}-U^{0}_{ij}).\]
Similarly, to construct the ADI scheme, adding a small perturbation term $R^{0,3}_{ij}:=\frac{\kappa^2\tau^2}{2}\delta_x^\alpha \delta_y^\alpha(U^{1}_{ij}-U^{0}_{ij})$ to \eqref{eq:before ADI initial} gives
\begin{equation}\label{eq:ADI with truncation error initial}
\begin{split}
\frac{U^{1}_{ij}-U^{0}_{ij}-\tau \partial_t U^0_{ij}}{\tau^2/2} =&-\kappa(\delta_x^\alpha+\delta_y^{\alpha})U^{1}_{ij} -\frac{\kappa^2\tau^2}{2}\delta_x^\alpha \delta_y^\alpha(U^{1}_{ij}-U^{0}_{ij})\\
 &-\kappa (L_h^\alpha-\delta_x^\alpha-\delta_y^{\alpha}) U^0_{ij} + g(U^0_{ij}) +R^{0,1}_{ij}+R^{0,2}_{ij}+R^{0,3}_{ij}.
\end{split}\end{equation}
By dropping the truncation error $R^{0,1}_{ij}+R^{0,2}_{ij}+R^{0,3}_{ij}$ and replacing the analytical solution $U^n_{ij}$ by the numerical approximation $u^n_{ij}$ with $n=0,1$, we obtain the S-ADI scheme for the fractional Laplacian wave equation \eqref{eq:wave_equation} at $t=\tau$:
\begin{equation}\label{eq:ADI scheme initial}
\begin{split}
\frac{u^{1}_{ij}-u^{0}_{ij}-\tau \partial_t u^0_{ij}}{\tau^2/2} =&-\kappa(\delta_x^\alpha+\delta_y^{\alpha})u^{1}_{ij} -\frac{\kappa^2\tau^2}{2}\delta_x^\alpha \delta_y^\alpha(u^{1}_{ij}-u^{0}_{ij})\\
 &-\kappa (L_h^\alpha-\delta_x^\alpha-\delta_y^{\alpha}) u^0_{ij} + g(u^0_{ij}),
\end{split}\end{equation}
where $\partial_t u^0_{ij}=\varphi_2(x_i,y_j)$ and $u^0_{ij}= \varphi_1(x_i,y_j)$.

\section{Fast implementation based on GS formula}\label{section:implementation}
It is worthy to mention that S-ADI schemes \eqref{eq:ADI scheme} and \eqref{eq:ADI scheme initial} can be equivalently rewritten to
\begin{equation}\label{eq:ADI scheme equivalent}
\left(1+\frac{\tau^2\kappa}{2}\delta_x^\alpha\right)\left(1+\frac{\tau^2\kappa}{2}\delta_y^\alpha\right)\hat u^n_{ij}=B^n_{ij},
\end{equation}
where
\[\hat u^n_{ij}:=\!\left\{
\begin{array}{ll}
u^{n+1}_{ij}-2u^{n}_{ij}+u^{n-1}_{ij},&n\ge1,\\
u^1_{ij}-u^0_{ij},&n=0,
\end{array}
\right. ~~ B^n_{ij}:=\!\left\{
\begin{array}{ll}
-\tau^2\kappa L_h^\alpha u^n_{ij}+\tau^2 g(u^n_{ij}),&n\ge1,\\
\tau\partial_t u^0_{ij}-\frac{\tau^2}{2}\kappa L_h^\alpha u^0_{ij}+\frac{\tau^2}{2} g(u^0_{ij}),&n=0.
\end{array}
\right.\]
To present \eqref{eq:ADI scheme equivalent} as matrix-vector format, we introduce some notations.
Let $\bB^{(n)}$, $\bU^{(n)}$, $\bA_x$, $\bA_y$, $\bI_x$ and $\bI_y$ be matrices in $\IR^{N\times N}$, 
$[\bB^{(n)}]_{ij}=B^n_{ij}$, $[\hat\bU^{(n)}]_{ij}=\hat u^n_{ij}$, $[\bA_x^{(n)}]_{ij}=[\bA_y^{(n)}]_{ij}=a^{(\alpha)}_{|i-j|}$, $\bI_x=\bI_y$ is the identity matrix. 

Then we transform \eqref{eq:ADI scheme equivalent}  into the matrix-vector format:
\begin{equation}\label{eq:ADI matrix vector}
  \left(\bI_y+ \frac{\tau^2\kappa}{2}\bA_y\right)\otimes\left(\bI_x+ \frac{\tau^2\kappa}{2}\bA_x\right) \MtoV(\hat\bU^{(n)})= \MtoV(\bB^{(n)}),   
\end{equation}
where $\MtoV(\cdot)$ is the vectorization of given matrix. 
To obtain right-hand side term $\bB^{(n)}$, one need compute $L_h^{\alpha}u^n_{ij}$, which is equivalent to calculate the product of a block Toeplitz matrix with Toeplitz block and a given vector. We compute it by embedding the matrix into a block-circulant-circulant-block matrix and using fast Fourier transform (FFT). For more detail, we refer to \cite{Hao_Zhang_2021}.

In following, we focus on how to solve the linear systems \eqref{eq:ADI matrix vector} by GS formula.
Thanks to the ADI technique, the coefficient matrix in \eqref{eq:ADI matrix vector} is the Kronecker product of two Toeplitz matrices. By the properties of Kronecker product, we have
\[\begin{split}
  \hat\bU^{(n)} =& \left(\bI_x+ \frac{\tau^2\kappa}{2}\bA_x\right)^{-1}\bB^{(n)}\left(\bI_y+ \frac{\tau^2\kappa}{2}\bA_y\right)^{-1} =\bH^{-1}[\bH^{-1}(\bB^{(n)})^\intercal]^\intercal,
\end{split}\]
where $\bH:=\bI_x+ \frac{\tau^2\kappa}{2}\bA_x$ is a symmetric positive definite (SPD) Toeplitz matrix. Hence for given vectors,  $\bH^{-1}\bv$ is required to be computed $2NM$ times in the whole computation process. In this situation, the popular iterative methods, such as preconditioned conjugate gradient method, are not efficient enough, because the coefficient matrix $\bH$ keeps unchanged. To utilize this characteristic more sufficiently, we introduce the following GS formula, which offers an explicit expression of the inverse of a SPD Toeplitz matrix multiplying a given vector.
\begin{proposition}[\cite{Pang_Sun_2011}]\label{propostition:GS formula}
Let $\bH$ be a SPD Toeplitz matrix and $\bc:=[p_1,p_2,\ldots,p_N]^\intercal=\bH^{-1}\bbe_1$,
 where $\bbe_1$ is the first column of identity matrix. Then for any given vector $\bv\in\IR^N$,
 \[ \bH^{-1} \bv = \Re(\hat \bv)+\bJ \Im(\hat \bv), \]
 where $\bJ$ is the anti-identity matrix, $\Re(\cdot)$ and $\Im(\cdot)$ present the real and imaginary parts, respectively, of given element, and 
$\hat \bv:= \frac{1}{2p_1}\bC\bS(\bv+\bi \bJ\bv)$,
 in which $\bC$ is the circulant matrix with $\bc$ as its first column, while $\bS$ is the skew-circulant matrix whose first column is $ \bs:=[p_1,-p_N,-p_{N-1},\ldots,-p_2]^\intercal$.
\end{proposition}

Note that the circulant matrix $\bC$ and the skew-circulant matrix $\bS$ have the following properties \cite{Lu_Pang_2018}:
\begin{equation}\label{eq:circulant diagonlizable}
 \bC=\bF^\contran\diag(\blam_c)\bF,\quad \bS = \bQ^\contran\bF^\contran\diag(\blam_s)\bF\bQ, 
 \end{equation}
where $\bF$ is the discrete Fourier matrix, $\bQ$ is the diagonal matrix with 
\[[1,\exp(-\bi\pi/N),\exp(-\bi 2\pi/N),\ldots,\exp(-\bi(N-1)\pi/N)]\]
 as its diagonal entries, and $\blam_c,\blam_s$ are vectors in $\IR^N$ containing the eigenvalues of $\bC$ and $\bS$. 
For a given vector $\bv\in\IR^N$, one could compute $\bF^\contran \bv$ and $\bF \bv$ by the FFTs $\ifft(\bv)$ and $\fft(\bv)$, respectively, with the complexity of $\mO(N\log N)$. Combining with \eqref{eq:circulant diagonlizable}, we compute $\blam_c$ and $\blam_s$ by the following formulas:
 \begin{equation}\label{eq:eigenvalue of CS}
 \blam_c = \fft(\bc),\quad \blam_s=\fft(\bQ\bs).
 \end{equation}
And if $\blam_c$ and $\blam_s$ have been obtained, the matrix-vector products of $\bC$ and $\bS$ with any given vector $\bv$ could be fast computed by
\begin{equation}\label{eq:CS matrix vector product}
\bC\bv=\ifft( \blam_c\circ \fft(\bv) ),\quad \bS\bv=\bQ^\contran \ifft(\blam_s\circ\fft(\bQ\bv)).
\end{equation}
With \eqref{eq:CS matrix vector product}, we state the practical implementation of GS formula as Algorithm \ref{algorithm:GS}. 
It is clear that only four FFTs of size to $N$ are required to compute $\bH^{-1}\bv$ based on the promise that $\blam_c,\blam_s,p_1$ have been obtained. 

\begin{algorithm}
\caption{GS formula for SPD Toeplitz}\label{algorithm:GS}
\begin{algorithmic}
\Statex \textbf{Input:} $\blam_c,\blam_s,p_1,\bv$
\State $\bv^{(1)} \gets \frac{1}{2p_1}(\bv+\bi\bJ\bv)$
\State $\bv^{(2)} \gets \bQ^\contran \ifft(\blam_s\circ\fft(\bQ\bv^{(1)}))$
\State $\bv^{(3)} \gets \ifft( \blam_c\circ \fft(\bv^{(2)}) )$
\State $\bH^{-1}\bv \gets \Re(\hat \bv^{(3)})+\bJ \Im(\hat \bv^{(3)})$
\Statex \textbf{Output:} $\bH^{-1}\bv$
\end{algorithmic}
\end{algorithm}

Although Algorithm \ref{algorithm:GS} needs to run $2NM$ times, $\bH^{-1}\be_1$ and \eqref{eq:eigenvalue of CS} are required to calculated only once throughout the implementation of S-ADI scheme to obtain $p_1,\blam_c,\blam_s$. And we compute $\bH^{-1}\be_1$ by the preconditioned conjugate gradient method with $\tau$ preconditioner whose complexity is only $\mO(N\log N)$. 

\section{Stability}\label{section:stability}

This section addresses the numerical stability of the S-ADI scheme proposed in Section \ref{section:construction}. In practice, unconditional stability is crucial for numerical schemes as it ensures that the computation will not fail even for large step sizes. 

We first introduce some notations and the inner products with their corresponding norms. 
Denote the index set, $\mI_h:=\{(i,j)|~1\le i\le N,1\le j\le N\}$.
Then a set of grid functions is defined by 
\[W_h:=\{\{w_{ij}|~(i,j)\in\IZ^2\}|~w_{ij}=0,~(i,j)\in\IZ^2 \setminus \mI_h\}.\] 
For the grid functions $w_1,w_2\in W_h$,  the inner product and norm  are defined by
\[\langle w_1,w_2\rangle := h^2\sum_{i=1}^{N}\sum_{j=1}^{N}(w_1)_{ij}(w_2)_{ij},\quad \|w_1\|:=\sqrt{\langle w_1,w_1\rangle}.\]
The next lemma indicates the operators $\delta_x^\alpha+\delta_y^\alpha$, $\delta_x^\alpha\delta_y^\alpha$ and $L_h^\alpha$ are all self adjoint and positive definite. 
\begin{lemma}[\cite{Lin_Ng_2019,Hao_Zhang_2021}]
$\langle (\delta_x^\alpha+\delta_y^\alpha)\cdot,\cdot\rangle$, $\langle \delta_x^\alpha\delta_y^\alpha\cdot,\cdot\rangle$ and $\langle L_h^\alpha\cdot,\cdot\rangle$ are all the inner products for the grid functions in $W_h$.
\end{lemma}
For ease of notation, we denote inner products and the corresponding norms as
\begin{gather*}
\langle\cdot,\cdot\rangle_{\tilde A} :=\langle (\delta_x^\alpha+\delta_y^\alpha)\cdot,\cdot\rangle  ,\quad
\langle\cdot,\cdot\rangle_{B} :=\langle \delta_x^\alpha\delta_y^\alpha\cdot,\cdot\rangle ,\quad
\langle\cdot,\cdot\rangle_A :=\langle L_h^\alpha\cdot,\cdot\rangle,\\
\|\cdot\|_{\tilde A}:=\sqrt{\langle \cdot,\cdot\rangle_{\tilde A}},\quad
\|\cdot\|_{B}:=\sqrt{\langle \cdot,\cdot\rangle_{B}},\quad
\|\cdot\|_A:=\sqrt{\langle \cdot,\cdot\rangle_A}.
\end{gather*}

The following lemma reveals the relationship between $\|\cdot\|_{\tilde A}$ and $\|\cdot\|_A$, which is critical for the unconditional stability of the S-ADI scheme.
\begin{lemma}[\cite{Li_Sun}]\label{lemma:norm_relationship_1}
For any grid function $w\in W_h$ and $\alpha\in(1,2)$, 
\[\|w\|_A^2 \le \|w\|_{\tilde A}^2.\]
\end{lemma}

The following lemma offers a manner to convert the estimation in $\|\cdot\|_A$ to the one in $\|\cdot\|$.
\begin{lemma}[\cite{Wang_Hao_2022}]\label{lemma:norm_relationship}
For any grid function $w\in W_h$ and $\alpha\in(1,2)$,
\[ \|w\|^2\le c_\alpha\|w\|_A^2 ,\]
 where $c_\alpha:=(\frac{\pi}{2})^\alpha\int_{\IR^2}\frac{1}{1+(\eta^2+\xi^2)^{\frac{\alpha}{2}}}\diff \eta \diff \xi $.
\end{lemma}

Let the sequence of grid functions $\{\check u^n:=\{\check u^n_{ij}|~(i,j)\in\IZ^2\}\in W_h\}_{1\le n\le M}$ be the result computed by S-ADI scheme \eqref{eq:ADI scheme} and \eqref{eq:ADI scheme initial} with the initial values $\check u^0:=\{\check u^0_{ij}|~(i,j)\in\IZ^2\},\partial_t \check u^0:=\{ \partial_t \check u^0_{ij}|~(i,j)\in\IZ^2\}\in W_h$. With above notations and lemmas, now we present the stability theorem of the S-ADI scheme.

\begin{theorem}\label{theorem:stability}
Suppose the function $g$ satisfies Lipschitz condition, i.e., there is a positive constant $L$ such that
\[ |g(x) - g(y)|\le L|x-y|, \quad \forall x,y\in\IR.  \]
Then for $1\le n\le M$,
\[\|u^n-\check u^n\|\le\sqrt{\frac{2 c_{\alpha}}{\kappa}}\exp\left(TL\sqrt{\frac{2 c_{\alpha}}{\kappa}}\right)
(2\|\partial_t u^0 -\partial_t \check u^0\| +\tau L\|u^0-\check u^0\| +\sqrt{2\kappa}\|u^0-\check u^0\|_A).\]
\end{theorem}
\begin{proof}
Denote $w^n:=\{w^n_{ij}:=u^n_{ij}-\check u^n_{ij}|~(i,j)\in\IZ^2\}$ for $0\le n \le M$, and $\partial_t w^0:=\{\partial_t w^0_{ij}:=\partial_t u^0_{ij}-\partial_t\check u^0_{ij}|~(i,j)\in\IZ^2\}$.
It follows from \eqref{eq:ADI scheme} for $n\ge1$ that
\begin{equation}\label{eq:stability proof 1}
  \begin{split}
&\frac{w^{n+1}_{ij}-2w^{n}_{ij}+w^{n-1}_{ij}}{\tau^2} +\kappa(\delta_x^\alpha+\delta_y^{\alpha})\frac{w^{n+1}_{ij}-2w^n_{ij}+w^{n-1}_{ij}}{2} \\
 &+\frac{\kappa^2\tau^2}{4}\delta_x^\alpha \delta_y^\alpha(w^{n+1}_{ij}-2w^{n}_{ij}+w^{n-1}_{ij})=-\kappa L_h^\alpha w^n_{ij} + g(u^n_{ij}) - g(\check u^n_{ij}).
\end{split}
\end{equation}
Multiply $h^2\frac{w^{n+1}_{ij} - w^{n-1}_{ij} }{\tau}$ on both sides of \eqref{eq:stability proof 1} and sum them for $(i,j)\in\mI_h$. Then we have
\begin{equation}\label{eq:stability proof 2}
\begin{split}
&\frac{\|\delta_t w^{n+\frac{1}{2}}\|^2 - \|\delta_t w^{n-\frac{1}{2}}\|^2}{\tau}\\
& + \frac{\tau\kappa}{2}(\|\delta_t w^{n+\frac{1}{2}}\|_{\tilde A}^2-\|\delta_t w^{n-\frac{1}{2}}\|_{\tilde A}^2) +\frac{\kappa^2\tau^3}{4}(\|\delta_t w^{n+\frac{1}{2}}\|_B^2-\|\delta_t w^{n-\frac{1}{2}}\|_B^2)\\
=&\underbrace{-\kappa\langle w^n,\frac{w^{n+1}-w^{n-1}}{\tau}\rangle_A}+\underbrace{h^2\sum_{i=1}^{N}\sum_{j=1}^{N}(g(u^n_{ij})-g(\check u^n_{ij}))(\delta_t w^{n+\frac{1}{2}}_{ij}+\delta_t w^{n-\frac{1}{2}}_{ij})},\\
& \hspace{2cm}=:I_1\hspace{4cm}=:I_2
\end{split}
\end{equation}
where notation
 $\delta_t w^{n\pm \frac{1}{2}}:=\frac{\pm w^{n\pm1}\mp w^{n} }{\tau}$
     has been used.
By simple calculation,
\begin{equation}\label{eq:stability proof 3}
\begin{split}
  I_1=&\frac{\kappa}{\tau}\langle w^n,w^{n+1}-w^{n-1}\rangle_A=-\frac{\kappa}{\tau}\langle w^n,w^{n+1}\rangle_A + \frac{\kappa}{\tau}\langle w^n,w^{n-1}\rangle_A\\
  =&\frac{\kappa}{2\tau}(\|w^{n+1}-w^n\|_A^2 - \|w^{n+1}\|_A^2 - \|w^n\|_A^2) \\
  &-\frac{\kappa}{2\tau}(\|w^n-w^{n-1}\|_A^2 - \|w^n\|_A^2 -\|w^{n-1}\|_A^2)\\
  =&\frac{\kappa\tau}{2}(\|\delta_t w^{n+\frac{1}{2}}\|_A^2-\|\delta_t w^{n-\frac{1}{2}}\|_A^2)-\kappa\frac{\|w^{n+1}\|_A^2 -\|w^{n-1}\|_A^2}{2\tau}.
\end{split}
\end{equation}
Denote
\[
\begin{split}
H_n:=&\left[\|\delta_t w^{n+\frac{1}{2}}\|^2 +\frac{\tau^2\kappa}{2}(\|\delta_t w^{n+\frac{1}{2}}\|_{\tilde A}^2 -\|\delta_t w^{n+\frac{1}{2}}\|_A^2) \right.\\
&\left.+\frac{\kappa}{2}(\|w^{n+1}\|_A^2 + \| w^n \|_A^2 ) +\frac{\kappa^2\tau^4}{4}\|\delta_t w^{n+\frac{1}{2}}\|_B^2 \right]^{\frac{1}{2}}.
\end{split}
\]
Lemma \ref{lemma:norm_relationship_1} indicates that $H_n$ is well-defined and real positive. With the notion of $H_n$ and by using equation \eqref{eq:stability proof 3}, \eqref{eq:stability proof 2} is simplified to
\begin{equation}\label{eq:stability proof 31}
  \frac{H_n^2-H_{n-1}^2}{\tau}=I_2,\quad 1\le n\le M-1.
\end{equation}
By Cauchy-Schwarz inequality,
\begin{equation}\label{eq:stability proof 4}
\begin{split}
I_2\le& h^2\sum_{i=1}^{N}\sum_{j=1}^{N} |g(u^n_ij) - g(\tilde u^n_{ij})|(|\delta w^{n+\frac{1}{2}}_{ij}|+|\delta w^{n-\frac{1}{2}}_{ij}|)\\
\le&Lh^2\sum_{i=1}^{N}\sum_{j=1}^{N} |w^n_{ij}|(|\delta w^{n+\frac{1}{2}}_{ij}|+|\delta w^{n-\frac{1}{2}}_{ij}|)
\le L\|w^n\|(\|\delta_t w^{n+\frac{1}{2}}\|+\|\delta_t w^{n-\frac{1}{2}}\|).
\end{split}
\end{equation}
Substituting  \eqref{eq:stability proof 4} into \eqref{eq:stability proof 31} gives that
\begin{equation}\label{eq:stability proof 5}
\frac{H_n^2-H_{n-1}^2}{\tau}\le L\|w^n\|(\|\delta_t w^{n+\frac{1}{2}}\|+\|\delta_t w^{n-\frac{1}{2}}\|),\quad 1\le n\le M-1.
\end{equation}
According to the definition of $H_n$ and Lemma \ref{lemma:norm_relationship}, we get
\begin{equation}\label{eq:stability proof 6}
\|w^n\|\le\sqrt{c_\alpha}\|w^n\|_A=\sqrt{\frac{2c_\alpha}{\kappa}}\sqrt{\frac{\kappa}{2}\|w^n\|_A^2} \le \sqrt{\frac{2c_\alpha}{\kappa}}\min\{H_{n-1},H_{n}\},
\end{equation}
and
\begin{equation}\label{eq:stability proof 7}
\|\delta_t w^{n+\frac{1}{2}}\| \le H_n,\quad \|\delta_t w^{n-\frac{1}{2}}\|\le H_{n-1}.
\end{equation}
Inserting \eqref{eq:stability proof 6} and \eqref{eq:stability proof 7} into \eqref{eq:stability proof 5}, we have
\[\frac{H_n^2-H_{n-1}^2}{\tau}\le L\sqrt{\frac{2c_\alpha}{\kappa}}H_{n-1}(H_n+H_{n-1}),\]
which implies
\begin{equation}\label{eq:stability proof 8}
\begin{split}
H_n\le&\left(1+\tau L\sqrt{\frac{2c_\alpha}{\kappa}}\right)H_{n-1}\le\exp\left(\tau L\sqrt{\frac{2c_\alpha}{\kappa}}\right)H_{n-1}\\
\le&\exp\left(TL\sqrt{\frac{2c_\alpha}{\kappa}}\right)H_0,\quad 1\le n\le M-1.
\end{split}\end{equation}

Now, we come to estimate $H_0$. According to \eqref{eq:ADI scheme initial}, we have
\begin{equation}\label{eq:stability proof 9}
\begin{split}
\frac{w^{1}_{ij}-w^{0}_{ij}-\tau \partial_t w^0_{ij}}{\tau^2/2} =&-\kappa(\delta_x^\alpha+\delta_y^{\alpha})(w^{1}_{ij}-w^0_{ij}) -\frac{\kappa^2\tau^2}{2}\delta_x^\alpha \delta_y^\alpha(w^{1}_{ij}-w^{0}_{ij})\\
 &-\kappa L_h^\alpha w^0_{ij} + g(u^0_{ij})-g(\check u^0_{ij}).
\end{split}\end{equation}
Multiply $h^2\frac{w^{1}_{ij} - w^{0}_{ij} }{\tau}$ on both sides of \eqref{eq:stability proof 9} and sum them for $(i,j)\in \mI_h$. Then we have
\begin{equation}\label{eq:stability proof 10}
\begin{split}
&\|\delta_t w^{\frac{1}{2}}\|^2+\frac{\kappa \tau^2}{2}\|\delta_t w^{\frac{1}{2}}\|^2_{\tilde A} +\frac{\kappa^2\tau^4}{4}\|\delta_t w^{\frac{1}{2}}\|^2_B +\frac{\tau\kappa}{2}\langle w^0,\delta_t w^{\frac{1}{2}}\rangle_A \\
=& \langle \partial_t w^0,\delta_t w^{\frac{1}{2}}\rangle  + \frac{\tau}{2}h^2\sum_{i=1}^{N}\sum_{j=1}^{N} (g(u^0_{ij})-g(\check u^0_{ij})) \delta_t w^{\frac{1}{2}}_{ij}.
\end{split}
\end{equation}
Substituting the following equality:
\[
\begin{split}
\frac{\tau\kappa}{2}\langle w^0,\delta_t w^{\frac{1}{2}}\rangle_A
=&\frac{\kappa}{4} (\|w^1\|_A^2-\|w^0\|_A^2-\tau^2\|\delta_t w^{\frac{1}{2}}\|_A^2)\\
=&\frac{\kappa}{4} (\|w^1\|_A^2+\|w^0\|_A^2-\tau^2\|\delta_t w^{\frac{1}{2}}\|_A^2)-\frac{\kappa}{2}\|w^0\|_A^2,
\end{split}
\]
into \eqref{eq:stability proof 10}, we get
\begin{equation}\label{eq:stability proof 11}
\begin{split}
&\frac{1}{2}(\|\delta_t w^{\frac{1}{2}}\|^2+\frac{\kappa \tau^2}{2}\|\delta_t w^{\frac{1}{2}}\|^2_{\tilde A} +\frac{\kappa^2\tau^4}{4}\|\delta_t w^{\frac{1}{2}}\|^2_B) +\frac{1}{2}H_0^2\\
=& \langle \partial_t w^0,\delta_t w^{\frac{1}{2}}\rangle  + \frac{\tau}{2}h^2\sum_{i=1}^{N}\sum_{j=1}^{N} (g(u^0_{ij})-g(\check u^0_{ij})) \delta_t w^{\frac{1}{2}}_{ij}+\frac{\kappa}{2}\|w^0\|_A^2.
\end{split}\end{equation}
Applying Cauchy-Swachtz inequality and \eqref{eq:stability proof 6}--\eqref{eq:stability proof 7} to \eqref{eq:stability proof 11}, we have
\[\begin{split}
\frac{1}{2}H_0^2\le& \langle \partial_t w^0,\delta_t w^{\frac{1}{2}}\rangle  + \frac{\tau}{2}h^2\sum_{i=1}^{N}\sum_{j=1}^{N} (g(u^0_{ij})-g(\tilde u^0_{ij})) \delta_t w^{\frac{1}{2}}_{ij}+\frac{\kappa}{2}\|w^0\|_A^2\\
\le&\|\partial_t w^0\|\|\delta_t w^{\frac{1}{2}}\|  +\frac{\tau}{2}L\|w^0\|\|\delta_t w^{\frac{1}{2}}\|+\frac{\kappa}{2}\|w^0\|_A^2\\
\le&\|\partial_t w^0\|H_0 +\frac{\tau}{2}L\|w^0\|H_0+\sqrt{\frac{\kappa}{2}}\|w^0\|_A H_0,\\
\end{split}\]
which indicates
\begin{equation}\label{eq:stability proof 12}
H_0\le 2\|\partial_t w^0\| +\tau L\|w^0\| +\sqrt{2\kappa}\|w^0\|_A.
\end{equation}
Inserting \eqref{eq:stability proof 12} into \eqref{eq:stability proof 8} and using \eqref{eq:stability proof 6}, we obtain
\[\|w^n\|\le\sqrt{\frac{2 c_{\alpha}}{\kappa}}\exp\left(TL\sqrt{\frac{2 c_{\alpha}}{\kappa}}\right)(2\|\partial_t w^0\| +\tau L\|w^0\| +\sqrt{2\kappa}\|w^0\|_A),\quad 1\le n\le M.\]
\end{proof}

The result of Theorem \ref{theorem:stability} means that, without restriction on the step size, the perturbation of the numerical solution of the S-ADI scheme is bounded by the perturbation of the initial condition.

\section{Error estimation}\label{section:error}
In this section, we  analyze the error of the S-ADI scheme. First, we present an estimation of the truncation error of the S-ADI scheme, which is the fundament for the global error analysis. To do this, we introduce some function spaces, and present some properties of the functions belonging to those spaces.

For $v\in L^1(\IR^d)$, we introduce its Fourier transform and inverse Fourier transform:
\[\mF(v)(\by):=\int_{\IR^d} v(\bx)\exp(-\bi \by\cdot\bx)\diff \bx,\quad \mF^{-1}(v)(\bx):=\frac{1}{(2\pi)^d}\int_{\IR^d} v(\by)\exp(\bi \bx\cdot\by)\diff \by,\]
where $\bx,\by\in\IR^d$, and $\bx\cdot\by$ denotes the Euclidean inner product in $\IR^d$. With the concept of Fourier transform, we introduce two function spaces:
\[\mL^{\gamma}(\IR^2)\!:=\!\left\{ v\in L^1(\IR^2)~\bigg| ~ \|v\|_{\mL^{\gamma}(\IR^2) } \!:=\!\int_{\IR^2} \left(1+\sqrt{\eta^2 +\xi^2}\right)^\gamma |\mF(v)(\eta,\xi)| \diff \eta \diff \xi <+\infty  \right\},\]
and
\[\mH^{\gamma}(\IR^2):=\left\{ v\in L^1(\IR^2)~\bigg| ~ \|v\|_{\mH^{\gamma}(\IR^2) } :=\int_{\IR^2} |\eta|^\gamma|\xi|^\gamma |\mF(v)(\eta,\xi)|\diff \eta \diff \xi <+\infty  \right\}.\]
The following lemma indicates that $\mL^{\alpha+2}(\IR^2)$ is a subspace of $\mL^{\alpha}(\IR^2)$ and $\mH^\alpha(\IR^2)$.
\begin{lemma}\label{lemma:localerr1}
Suppose $\alpha\in(1,2)$ and $v\in \mL^{\alpha+2}(\IR^2)$. Then $v\in\mL^{\alpha}(\IR^2)\cap\mH^\alpha(\IR^2)$, and
\[\|v\|_{\mL^\alpha(\IR^2)} \le \|v\|_{\mL^{2+\alpha}(\IR^2)},\quad \|v\|_{\mH^\alpha(\IR^2)}\le\frac{1}{2}\|v\|_{\mL^{2+\alpha}(\IR^2)}.\]
\end{lemma}
\begin{proof}
\[(1+\sqrt{\eta^2 +\xi^2})^\alpha \le  (1+\sqrt{\eta^2 +\xi^2})^{2+\alpha},\quad \forall (\eta,\xi)\in\IR^2,\]
indicates $v\in\mL^{\alpha}(\IR^2)$ and the first inequality holds, while
\[
\begin{split}&|\eta|^\alpha|\xi|^\alpha\le \frac{1}{2}(|\eta|^{2\alpha} + |\xi|^{2\alpha})\le\frac{1}{2}(|\eta|^{2} + |\xi|^{2})^\alpha=\frac{1}{2}(\sqrt{|\eta|^{2} + |\xi|^{2}})^{2\alpha}\\
\le&\frac{1}{2}(1+\sqrt{|\eta|^{2} + |\xi|^{2}})^{2\alpha}\le\frac{1}{2}(1+\sqrt{|\eta|^{2} + |\xi|^{2}})^{2+\alpha},\quad \forall (\eta,\xi)\in\IR^2,\end{split}\]
implies $v\in\mH^{\alpha}(\IR^2)$ and the second inequality holds.
\end{proof}

Let $\{\tilde a^{(\alpha)}_{ij}\}$ be the Fourier coefficients of function $\left(4\sin^2\left(\frac{\eta}{2}\right)\right)^{\frac{\alpha}{2}}  \!+\! \left(4\sin^2\left(\frac{\xi}{2}\right) \right)^{\frac{\alpha}{2}}$, i.e.,
\[\tilde a^{(\alpha)}_{ij} :=\frac{1}{4\pi^2}\int_{-\pi}^{\pi}\int_{-\pi}^{\pi} \left[\left(4\sin^2\left(\frac{\eta}{2}\right)\right)^{\frac{\alpha}{2}}  + \left(4\sin^2\left(\frac{\xi}{2}\right) \right)^{\frac{\alpha}{2}}\right] \exp(-\bi (i\eta+j\xi))\diff \eta \diff \xi.\]
By simple calculation, we  find that
\[\tilde a^{(\alpha)}_{ij}=\left\{
\begin{array}{ll}
2a^{(\alpha)}_{0},& i=0,~j=0,\\
a^{(\alpha)}_{j},& i=0,~j\neq0,\\
a^{(\alpha)}_{i},& i\neq0,~j=0,\\
0,& i\neq0,~j\neq0,\\
\end{array}
\right.\]
which with \eqref{eq:single_discrete_Riesz} and \eqref{eq:single coefficient discrete Riesz} indicates that
\[(\delta_x^\alpha +\delta_y^\alpha)v(x,y)=\frac{1}{h^\alpha}\sum_{i,j\in\IZ}\tilde a^{(\alpha)}_{ij} v(x+ih,y+jh).\]
With above representation of the discrete Riesz fractional operator $\delta_x^\alpha +\delta_y^\alpha$, we obtain the following lemma, which means that $(\delta_x^\alpha+\delta_y^\alpha)v$ is bounded if the function $v$ possesses a certain regularity.

\begin{lemma}\label{lemma:boundness of discrete Riesz}
Suppose $\alpha\in(1,2)$ and $v\in\mL^\alpha(\IR^2)\cap \mC(\IR^2)$. Then
\[ |(\delta_x^\alpha+\delta_y^\alpha)v(x,y)|  \le  \frac{2^{1-\frac{\alpha}{2}}}{4\pi^2}\|v\|_{\mL^{\alpha}(\IR^2)},\quad (x,y)\in\IR^2.\]
\end{lemma}
\begin{proof}
Taking Fourier transform of $(\delta_x^\alpha+\delta_y^\alpha)v$ and according the definition of $\{\tilde a^{(\alpha)}_{ij}\}$, we have
\[
\begin{split}
&h^\alpha \mF((\delta_x^\alpha+\delta_y^\alpha)v)(\eta,\xi)\\
=&\sum_{i,j\in\IZ}\tilde a^{(\alpha)}_{ij}\int_{\IR^2} v(x+ih,y+jh)\exp(-\bi(\eta x+\xi y))\diff x \diff  y\\
=&\sum_{i,j\in\IZ}\tilde a^{(\alpha)}_{ij} \exp(\bi(ih\eta+jh\xi))\mF(v)(\eta,\xi)\\
=&\left[ \left(4\sin^2\left(\frac{h\eta}{2}\right)\right)^{\frac{\alpha}{2}}  + \left(4\sin^2\left(\frac{h\xi}{2}\right) \right)^{\frac{\alpha}{2}} \right]\mF(v)(\eta,\xi).
\end{split}\]
Taking inverse Fourier transform on above equation, we have
\begin{equation}\label{eq:proof_boundness of discrete Riesz 1}
\begin{split}
&h^\alpha(\delta_x^\alpha+\delta_y^\alpha)v(x,y)=\mF^{-1}(\mF(h^\alpha(\delta_x^\alpha+\delta_y^\alpha)v))(x,y)\\
=&\frac{1}{4\pi^2}\int_{\IR^2} \left[ \left(4\sin^2\left(\frac{h\eta}{2}\right)\right)^{\frac{\alpha}{2}}  + \left(4\sin^2\left(\frac{h\xi}{2}\right) \right)^{\frac{\alpha}{2}} \right]\mF(v)(\eta,\xi)\exp(\bi(x\eta+ y\xi))\diff \eta \diff \xi.
\end{split}
\end{equation}
Inserting the following inequality
\[
\begin{split}
&\left(4\sin^2\left(\frac{h\eta}{2}\right)\right)^{\frac{\alpha}{2}}  + \left(4\sin^2\left(\frac{h\xi}{2}\right) \right)^{\frac{\alpha}{2}}\le2^{1-\frac{\alpha}{2}} \left(4\sin^2\left(\frac{h\eta}{2}\right) + 4\sin^2\left(\frac{h\xi}{2}\right) \right)^{\frac{\alpha}{2}}\\
\le&2^{1-\frac{\alpha}{2}} h^\alpha (\eta^2+\xi^2)^{\frac{\alpha}{2}}\le2^{1-\frac{\alpha}{2}} h^\alpha (1+\sqrt{\eta^2+\xi^2})^{\alpha},
\end{split}\]
into \eqref{eq:proof_boundness of discrete Riesz 1} yields that
\[\begin{split}
&|(\delta_x^\alpha+\delta_y^\alpha)v(x,y)|
\!\le\!\frac{1}{4\pi^2 h^\alpha}\!\int_{\IR^2} \!\left[ \left(\!4\sin^2\left(\frac{h\eta}{2}\right)\right)^{\frac{\alpha}{2}} \!\!\! +\! \left(\!4\sin^2\left(\frac{h\xi}{2}\right) \right)^{\frac{\alpha}{2}} \right]\!|\mF(v)(\eta,\xi)|\diff \eta \diff \xi\\
\le&\frac{2^{1-\frac{\alpha}{2}}}{4\pi^2}\int_{\IR^2}(1+\sqrt{\eta^2+\xi^2})^{\alpha}|\mF(v)(\eta,\xi)|\diff \eta \diff \xi
=\frac{2^{1-\frac{\alpha}{2}}}{4\pi^2}\|v\|_{\mL^\alpha(\IR^2)},
\end{split}
\]
which completes the proof.
\end{proof}

Let $\{\hat a^{(\alpha)}_{ij}\}$ be the Fourier coefficients of function $[4\sin^2(\frac{\eta}{2})]^{\frac{\alpha}{2}} [4\sin^2(\frac{\xi}{2}) ]^{\frac{\alpha}{2}}$, i.e., 
\[\hat a^{(\alpha)}_{ij} :=\frac{1}{4\pi^2}\int_{-\pi}^{\pi}\int_{-\pi}^{\pi} [4\sin^2(\frac{\eta}{2})]^{\frac{\alpha}{2}}   [4\sin^2(\frac{\xi}{2}) ]^{\frac{\alpha}{2}} \exp(-\bi (i\eta+j\xi))\diff \eta \diff \xi.\]
According to the definition of the coefficient $\{a^{(\alpha)}_{i}\}$ in \eqref{eq:single coefficient discrete Riesz}, it is easy to verify that $\hat a^{(\alpha)}_{ij}=a^{(\alpha)}_i a^{(\alpha)}_j$, for $i,j\in \IZ$. Combining this equation with the definition of operator $\delta^\alpha_x\delta^\alpha_y$, we have 
\[\delta^\alpha_x\delta^\alpha_y v(x,y)=\frac{1}{h^{2\alpha}}\sum_{i,j\in\IZ}\hat a^{(\alpha)}_{ij} v(x+ih,y+jh).\]
Then following the similar ways in the proof of Lemma \ref{lemma:boundness of discrete Riesz} gives the following bound of $\delta_x^\alpha\delta_y^\alpha v$ if the function $v$ satisfies a certain regularity assumption.
\begin{lemma}\label{lemma:boundness of ADI reminder}
Suppose $\alpha\in(1,2)$ and $v\in\mH^\alpha(\IR^2)\cap \mC(\IR^2)$, $\mF v\in L^1(\IR^2)$. Then
\[ |\delta_x^\alpha\delta_y^\alpha v(x,y)|  \le  \frac{1}{4\pi^2}\|v\|_{\mH^{\alpha}(\IR^2)},\quad (x,y)\in\IR^2.\]
\end{lemma}
\begin{proof}
By similar techniques used in the proof of \eqref{lemma:boundness of discrete Riesz}, we have
\begin{equation}\label{eq:proof boundness of ADI reminder 1}
\begin{split}
&h^{2\alpha}\delta_x^\alpha\delta_y^\alpha v(x,y)=\mF^{-1}(\mF(h^{2\alpha}\delta_x^\alpha\delta_y^\alpha v))(x,y)\\
=&\frac{1}{4\pi^2}\int_{\IR^2}  \left(4\sin^2\left(\frac{h\eta}{2}\right)\right)^{\frac{\alpha}{2}}  \left(4\sin^2\left(\frac{h\xi}{2}\right) \right)^{\frac{\alpha}{2}} \mF(v)(\eta,\xi)\exp(\bi(x\eta+ y\xi))\diff \eta \diff \xi.
\end{split}
\end{equation}
Applying the inequality $\left(4\sin^2\left(\frac{h\eta}{2}\right)\right)^{\frac{\alpha}{2}}  \left(4\sin^2\left(\frac{h\xi}{2}\right) \right)^{\frac{\alpha}{2}}\le h^{2\alpha}|\eta|^\alpha|\xi|^\alpha$ to \eqref{eq:proof boundness of ADI reminder 1}, we get
\[|\delta_x^\alpha\delta_y^\alpha v(x,y)|\le \frac{1}{4\pi^2}\int_{\IR^2}|\eta|^\alpha|\xi|^\alpha|\mF(v)(\eta,\xi)|\diff \eta \diff \xi
=\frac{1}{4\pi^2}\|v\|_{\mH^\alpha(\IR^2)}. \]
This completes the proof.
\end{proof}

We also need the following lemma, which gives an estimation of the local error when applying \eqref{eq:center_disccritzation_fractional_Laplacian} to approximate the fractional Laplacian.
\begin{lemma}[\cite{Wang_Hao_2022}]\label{lemma:laplacian_local error}
Suppose $v\in\mL^{2+\alpha}(\IR^2)\cap\mC(\IR^2)$ and $\alpha\in(1,2)$. There exist a positive constant $c^{(\alpha)}$ independent of $h$ such that
\[| (-\Delta)^{\alpha}v(x,y) - L_h^{\alpha}v(x,y) |\le c^{(\alpha)}h^2\|v\|_{\mL^{2+\alpha}(\IR^2) },\quad (x,y)\in\IR^2.\]
\end{lemma}


Let $R^n_{ij}:=R^{n,1}_{ij}+R^{n,2}_{ij}+R^{n,3}_{ij}$ for $0\le n\le M$, and
 \[
 \begin{split}\check c:=\max&\left\{ \kappa c^{(\alpha)}\max_{t\in[0,T]}\|u(\cdot,\cdot,t)\|_{\mL^{2+\alpha}(\IR^2)},
 \frac{\kappa^2}{8\pi^2}\max_{t\in[0,T]}\|u(\cdot,\cdot,t)\|_{\mL^{2+\alpha}(\IR^2)},\right.\\
 &\left.\frac{1}{3}\max_{t\in[0,T]}\|\partial_t^3u(\cdot,\cdot,t)\|_{L^{\infty}(\IR^2)}+\kappa\max_{t\in[0,T]}\|\partial_t u(\cdot,\cdot,t)\|_{\mL^{2+\alpha}(\IR^2)}
 \right\},
 \end{split}\]
\[
\begin{split}
\tilde c:=\max&
\left\{
\kappa c^{(\alpha)}\max_{t\in[0,T]}\|u(\cdot,\cdot,t)\|_{\mL^{2+\alpha}(\IR^2)},  \frac{1}{12}\max_{t\in[0,T]}\|\partial_t^4u(\cdot,\cdot,t)\|_{L^{\infty}(\IR^2)} \right.  \\
&\left. +\frac{\kappa}{2}\max_{t\in[0,T]}\|\partial_t^2u(\cdot,\cdot,t)\|_{\mL^{2+\alpha}(\IR^2)}
+\frac{\kappa^2}{8\pi^2}\max_{t\in[0,T]}\|u(\cdot,\cdot,t)\|_{\mL^{2+\alpha}(\IR^2)}
    \right\}.
\end{split}\]
With above notations and lemmas, we present the estimation of the truncation errors of the S-ADI scheme.
\begin{lemma}\label{lemma:truncation error}
Suppose $u\in\mC^{4}(0,T;\mL^{2+\alpha}(\IR^2)\cap\mC(\IR^2))$. Then
\[|R^{0}_{ij}| \le \check c(\tau+\tau^2+h^2);\quad |R^{n}_{ij}| \le \tilde c(\tau^2+h^2),\quad 1\le n \le M.\]
\end{lemma}
\begin{proof}
According to $\left|R^n_{ij}\right|\le\left|R^{n,1}_{ij}\right|+\left|R^{n,2}_{ij}\right|+\left|R^{n,3}_{ij}\right|$, we next estimate the three terms on the right-hand side.
By Lemma \ref{lemma:laplacian_local error}, we have
\[\left|R^{n,1}_{ij}\right|\le \kappa c^{(\alpha)}h^2\|u(\cdot,\cdot,t_n)\|_{\mL^{2+\alpha}(\IR^2)}\le \kappa c^{(\alpha)}h^2\max_{t\in[0,T]}\|u(\cdot,\cdot,t)\|_{\mL^{2+\alpha}(\IR^2)},\quad 0\le n \le M.\]
Using Lemma \ref{lemma:localerr1} and Lemma \ref{lemma:boundness of ADI reminder}, we get
\[
\begin{split}
\left|R^{0,3}_{ij}\right|\le& \frac{\kappa^2\tau^2}{2}(|\delta_x^\alpha\delta_y^\alpha u(x_i,y_j,t_1)|+|\delta_x^\alpha\delta_y^\alpha u(x_i,y_j,t_0)|)\\
\le&\frac{\kappa^2\tau^2}{2}\left(\frac{1}{4\pi^2}\|u(\cdot,\cdot,t_1)\|_{\mH^\alpha(\IR^2)} +\frac{1}{4\pi^2}\|u(\cdot,\cdot,t_0)\|_{\mH^\alpha(\IR^2)} \right)\\
\le&\frac{\kappa^2\tau^2}{4\pi^2}\max_{t\in[0,T]}\|u(\cdot,\cdot,t)\|_{\mH^\alpha(\IR^2)} \le \frac{\kappa^2\tau^2}{8\pi^2}\max_{t\in[0,T]}\|u(\cdot,\cdot,t)\|_{\mL^{2+\alpha}(\IR^2)}.
\end{split}\]
Similarly,
\[\left|R^{n,3}_{ij}\right|\le\frac{\kappa^2\tau^2}{8\pi^2}\max_{t\in[0,T]}\|u(\cdot,\cdot,t)\|_{\mL^{2+\alpha}(\IR^2)},\quad 1\le n\le M.\]
Applying Taylor formula and Lemmas \ref{lemma:boundness of discrete Riesz}--\ref{lemma:localerr1} to $R^{0,2}_{ij}$ yields
\[
\begin{split}
\left|R^{0,2}_{ij}\right|
\le& \left|\frac{2}{\tau^2}\times\frac{1}{2!}\int_{0}^{\tau}\partial_t^3 u(x_i,y_j,\eta)(\tau-\eta)^2 \diff \eta\right|
+\kappa\left|(\delta_x^\alpha+\delta_y^\alpha)\int_0^\tau \partial_t u(x_i,y_j,\eta)\diff  \eta\right|\\
\le&\frac{1}{\tau^2}\max_{t\in[0,T]}\|\partial_t^3u(\cdot,\cdot,t)\|_{L^{\infty}(\IR^2)}\int_0^{\tau}(\tau-\eta)^2\diff \eta
+\tau\kappa\max_{t\in[0,T]}\|\partial_t u(\cdot,\cdot,t)\|_{\mL^{\alpha}(\IR^2)}\\
\le & \left(\frac{1}{3}\max_{t\in[0,T]}\|\partial_t^3u(\cdot,\cdot,t)\|_{L^{\infty}(\IR^2)}+\kappa\max_{t\in[0,T]}\|\partial_t u(\cdot,\cdot,t)\|_{\mL^{2+\alpha}(\IR^2)}\right)\tau.
\end{split}
\]
Similarly, for $1\le n\le M$, we have
{\small
\[\begin{split}
\left|R^{n,2}_{ij}\right|\!\le\!
&\left| \frac{1}{\tau^2}\left(  \frac{1}{3!}\int_0^\tau \!\partial_t^4 u(x_i,y_j,t_n\!+\!\eta) (\tau\!-\!\eta)^3 \diff \eta
+\!\frac{1}{3!}\int_0^{-\tau} \!\partial_t^4 u(x_i,y_j,t_n\!+\!\eta) (-\!\tau\!-\!\eta)^3 \diff \eta
\right)  \right|\\
&+\left| \frac{\kappa}{2}(\delta_x^\alpha\!+\!\delta_y^\alpha)\left(  \int_0^\tau \partial_t^2 u(x_i,y_j,t_n\!+\!\eta) (\tau\!-\!\eta) \diff \eta
+\!\!\int_0^{-\tau}\!\! \partial_t^2 u(x_i,y_j,t_n\!+\!\eta) (-\!\tau\!-\!\eta) \diff \eta
\right)  \right|\\
\le&\frac{1}{6\tau^2}\max_{t\in[0,T]}\|\partial_t^4 u(\cdot,\cdot,t)\|_{L^{\infty}(\IR^2)}\left( \int_0^\tau (\tau-\eta)^3\diff \eta + \int_0^{-\tau} (-\tau-\eta)^3\diff \eta\right)\\
&+\frac{\kappa}{2}\max_{t\in[0,T]}\|\partial_t^2 u(\cdot,\cdot,t)\|_{\mL^{\alpha}(\IR^2)}\left( \int_0^\tau (\tau-\eta)\diff \eta + \int_0^{-\tau} (-\tau-\eta)\diff \eta\right)\\
\le&\left(\frac{1}{12}\max_{t\in[0,T]}\|\partial_t^4 u(\cdot,\cdot,t)\|_{L^{\infty}(\IR^2)} + \frac{\kappa}{2}\max_{t\in[0,T]}\|\partial_t^2 u(\cdot,\cdot,t)\|_{\mL^{2+\alpha}(\IR^2)}\right)\tau^2.
\end{split}\]}
Combing above inequalities completes the proof.
\end{proof}


After estimating the truncation error carefully, we come to give the estimation of the global error of the S-ADI scheme.

\begin{theorem}\label{theorem:convergence}
Suppose the solution of the initial boundary value problem \eqref{eq:wave_equation}--\eqref{eq:initial condition 2} $u\in\mC^{4}(0,T;\mL^{2+\alpha}(\IR^2)\cap\mC(\IR^2))$. Then the error of S-ADI scheme \eqref{eq:ADI scheme} and \eqref{eq:ADI scheme initial} has the following estimation for $1\le n\le M$:
\[\|e^n\|\le C_1(\tau^2+\tau^3+\tau h^2) + C_2(\tau^2+h^2), \]
where $e^n:=\{e^n_{ij}:=U^n_{ij}-u^n_{ij}~|~(i,j)\in \IZ^2\}$, 
\[C_1:=\sqrt{\frac{2c_\alpha}{\kappa}}\exp\left(TL\sqrt{\frac{2c_\alpha}{\kappa}}\right)\frac{\check c(b\!-\!a)}{2},\quad 
C_2:=\sqrt{\frac{2c_\alpha}{\kappa}}\phi\left(TL\sqrt{\frac{2c_\alpha}{\kappa}}\right)T\tilde c(b\!-\!a),\]
and $\phi(x)=(\exp(x)-1)/x$.
\end{theorem}
\begin{proof}
Subtracting \eqref{eq:ADI scheme} from \eqref{eq:ADI with truncation error}, we have
\begin{equation}\label{eq:convergence proof1}
  \begin{split}
&\frac{e^{n+1}_{ij}-2e^{n}_{ij}+e^{n-1}_{ij}}{\tau^2} \!+\kappa(\delta_x^\alpha+\delta_y^{\alpha})\frac{e^{n+1}_{ij}-2e^n_{ij}+e^{n-1}_{ij}}{2} +\frac{\kappa^2\tau^2}{4}\delta_x^\alpha \delta_y^\alpha(e^{n+1}_{ij}-2e^{n}_{ij}+e^{n-1}_{ij})\\
 &=-\kappa L_h^\alpha e^n_{ij} + g(U^n_{ij}) - g(u^n_{ij}) + R^n_{ij},\quad n\ge1.
\end{split}
\end{equation}
Multiply $h^2\frac{e^{n+1}_{ij} - e^{n-1}_{ij} }{\tau}$ on both sides of \eqref{eq:convergence proof1} and sum them for $(i,j)\in\mI_h$. Then we have
\begin{equation}\label{eq:convergence proof2}
\begin{split}
&\frac{\|\delta_t e^{n+\frac{1}{2}}\|^2 - \|\delta_t e^{n-\frac{1}{2}}\|^2}{\tau} \\
&+ \frac{\tau\kappa}{2}(\|\delta_t e^{n+\frac{1}{2}}\|_{\tilde A}^2-\|\delta_t e^{n-\frac{1}{2}}\|_{\tilde A}^2) +\frac{\kappa^2\tau^3}{4}(\|\delta_t e^{n+\frac{1}{2}}\|_B^2-\|\delta_t e^{n-\frac{1}{2}}\|_B^2)\\
=&-\kappa\langle e^n,\delta_t e^{n+\frac{1}{2}}+\delta_t e^{n-\frac{1}{2}}\rangle_A+h^2\sum_{i=1}^{N}\sum_{j=1}^{N}(g(U^n_{ij})-g(u^n_{ij}))(\delta_t e^{n+\frac{1}{2}}_{ij}+\delta_t e^{n-\frac{1}{2}}_{ij})\\
&+\langle R^n,\delta_t e^{n+\frac{1}{2}}+\delta_t e^{n-\frac{1}{2}}\rangle,
\end{split}
\end{equation}
where notation $\delta_t e^{n\pm \frac{1}{2}}:=\frac{\pm e^{n\pm1} \mp e^{n} }{\tau}$ has been used.
Denote
\[
\begin{split}
E_n:=&\left[\|\delta_t e^{n+\frac{1}{2}}\|^2 +\frac{\tau^2\kappa}{2}(\|\delta_t e^{n+\frac{1}{2}}\|_{\tilde A}^2 -\|\delta_t e^{n+\frac{1}{2}}\|_A^2) \right.\\
&\left.+\frac{\kappa}{2}(\|e^{n+1}\|_A^2 + \| e^n \|_A^2 ) +\frac{\kappa^2\tau^4}{4}\|\delta_t e^{n+\frac{1}{2}}\|_B^2 \right]^{\frac{1}{2}}.
\end{split}
\]
Applying the similar techniques used in \eqref{eq:stability proof 3} and \eqref{eq:stability proof 4} to the right-hand side of \eqref{eq:convergence proof2} yields
\begin{equation}\label{eq:convergence proof3}
\frac{E_n^2\!-\!E_{n-1}^2}{\tau}\!\le\! L\|e^n\|(\|\delta_t e^{n+\frac{1}{2}}\|\!+\!\|\delta_t e^{n-\frac{1}{2}}\|)\!+\!\|R^n\|(\|\delta_t e^{n+\frac{1}{2}}\|\!+\!\|\delta_t e^{n-\frac{1}{2}}\|),~~ 1\!\le \!n\!\le\! M\!-\!1.
\end{equation}
According to the definition of $E_n$ and Lemma \ref{lemma:norm_relationship}, we get
\begin{equation}\label{eq:convergence proof4}
\|e^n\|\le\sqrt{c_\alpha}\|e^n\|_A=\sqrt{\frac{2c_\alpha}{\kappa}}\sqrt{\frac{\kappa}{2}\|e^n\|_A^2} \le \sqrt{\frac{2c_\alpha}{\kappa}}E_{n-1},
\end{equation}
and
\begin{equation}\label{eq:convergence proof5}
\|\delta_t e^{n+\frac{1}{2}}\| \le E_n,\quad \|\delta_t e^{n-\frac{1}{2}}\|\le E_{n-1}.
\end{equation}
Inserting \eqref{eq:convergence proof4} and \eqref{eq:convergence proof5} into \eqref{eq:convergence proof3} and using Lemma \ref{lemma:truncation error}, we have
\[
\begin{split}
\frac{E_n^2-E_{n-1}^2}{\tau}
\le &\left(L\sqrt{\frac{2c_\alpha}{\kappa}}E_{n-1}+\|R^n\|\right)(E_n+E_{n-1})\\
\le & \left(L\sqrt{\frac{2c_\alpha}{\kappa}}E_{n-1}+\tilde c(b-a)(\tau^2+h^2)\right)(E_n+E_{n-1}) ,\quad 1\le n \le M-1,
\end{split}
\]
which with Lemma \ref{lemma:truncation error} indicates
\begin{equation}\label{eq:convergence proof6}
\begin{split}
E_n\le&\left(1+\tau L\sqrt{\frac{2c_\alpha}{\kappa}}\right)E_{n-1} +\tau \tilde c(b-a)(\tau^2+h^2) \\
\le&\left(1+\tau L\sqrt{\frac{2c_\alpha}{\kappa}}\right)^2E_{n-2} + \left[1+\left(1+\tau L\sqrt{\frac{2c_\alpha}{\kappa}}\right)\right]\tau \tilde c(b-a)(\tau^2+h^2)\\
&\cdots\\
\le&\left(1+\tau L\sqrt{\frac{2c_\alpha}{\kappa}}\right)^nE_{0}+\sum_{k=0}^{n-1}\left(1+\tau L\sqrt{\frac{2c_\alpha}{\kappa}}\right)^k \tau \tilde c(b-a)(\tau^2+h^2)\\
\le& \exp\left(TL\sqrt{\frac{2c_\alpha}{\kappa}}\right)E_0 + \phi\left(TL\sqrt{\frac{2c_\alpha}{\kappa}}\right)T\tilde c(b-a)(\tau^2+h^2).
\end{split}
\end{equation}
The rest work is to estimate $E_0$. According to $e^0_{ij}=0$, we have
\[\begin{split}
E^2_0=&\frac{1}{\tau^2}\|e^1\|^2+\frac{\kappa}{2}(\|e^1\|_{\tilde A}^2-\|e^1\|_A^2)+\frac{\kappa}{2}\|e^1\|_A^2+\frac{\kappa^2\tau^2}{4}\|e^1\|_B^2\\
=&\frac{1}{\tau^2}\|e^1\|^2+\frac{\kappa}{2}\|e^1\|_{\tilde A}^2+\frac{\kappa^2\tau^2}{4}\|e^1\|_B^2.
\end{split}\]
Subtracting \eqref{eq:ADI scheme initial} from \eqref{eq:ADI with truncation error initial} gives
\begin{equation}\label{eq:convergence proof7}
\frac{2e^1_{ij}}{\tau^2}+\kappa(\delta_x^\alpha+\delta_y^\alpha)e^1_{ij}+\frac{\kappa^2\tau^2}{2}\delta_x^\alpha\delta_y^\alpha e^1_{ij}=R^0_{ij}.
\end{equation}
Multiply $h^2\frac{e^{1}_{ij}}{2}$ on both sides of \eqref{eq:convergence proof7} and sum them for $(i,j)\in\mI_h$. Then we have
\[
\begin{split}
&\frac{1}{\tau^2}\|e^1\|^2+\frac{\kappa}{2}\|e^1\|_{\tilde A}^2+\frac{\kappa^2\tau^2}{4}\|e^1\|_B^2\\
=&\langle\frac{\tau}{2}R^0,\frac{1}{\tau}e^1\rangle \le\frac{\tau}{2} \|R^0\| \frac{1}{\tau}\|e^1\|
\le\frac{\tau}{2}\check c(b-a)(\tau+\tau^2+h^2)E_0,
\end{split}\]
which implies that
\begin{equation}\label{eq:convergence proof8}
E^0\le\frac{\check c}{2}(b-a)(\tau^2+\tau^3+\tau h^2).
\end{equation}
Inserting \eqref{eq:convergence proof8} into \eqref{eq:convergence proof6} and using \eqref{eq:convergence proof4} complete the proof.
\end{proof}

\begin{remark}
If we assume that $\tau\le1$, the result of Theorem \ref{theorem:convergence} could be reformed as
\[\|e^n\| \le \max\{2C_1,C_2\}(\tau^2+h^2),\quad 1\le n\le M.\]
This means the S-ADI scheme for the fractional Laplacian wave equation is second-order convergent.
\end{remark}


\section{Numerical experiments}\label{section:numerical}
In this section, we  use the S-ADI scheme to solve the fractional Laplacian sine-Gordon equation and Klein-Gordon equation, and compare it with the following traditional non-ADI scheme:
\begin{numcases}{}
\frac{u^{n+1}_{ij}-2u^{n}_{ij}+u^{n-1}_{ij}}{\tau^2} = -\frac{\kappa}{2} L_h^\alpha (u^{n+1}_{ij}+u^{n-1}_{ij}) + g(u^n_{ij}), \quad n\ge 1,\label{eq:non-ADI scheme1}\\
\frac{u^{1}_{ij}-u^{0}_{ij}-\tau \partial_t u^0_{ij}}{\tau^2/2} =-\kappa L_h^\alpha u^1_{ij} + g(u^0_{ij}).\label{eq:non-ADI scheme2}
\end{numcases}
The coefficient matrices of the linear systems appearing in the S-ADI scheme are Toeplitz. Hence, we  apply the Gohberg–Semencul formula discussed in Section \ref{section:implementation} to solve them. For the linear systems in the non-ADI scheme \eqref{eq:non-ADI scheme1}--\eqref{eq:non-ADI scheme2}, we use the preconditioned conjugate gradient method with the state of the art $\tau$-preconditioner \cite{Li_Sun,Huang_Lin_2022}.

Denote the numerical solution at $(x_i,y_j,t_n)$ computed with step sizes $\tau,h$ by $u_{ij}^n(\tau,h)$ for $1\le n\le M$ and $1\le i,j \le N$. 
When verifying the time convergence order, the error and convergence order are estimated by the following formula:
\begin{gather*}\text{Error}_1(\tau,h):=  \sqrt{h^2\sum_{ij}|u^M_{ij}(\tau,h)-u^{2M}_{ij}(\tau/2,h)|^2},\\  
\text{Order}_1(\tau,h):= \log_2(\text{Error}_1(\tau,h)/\text{Error}_1(\tau/2,h)),
\end{gather*}
where $h$ is a small fixed positive number. 
When verifying the space convergence order, the error and convergence order are estimated by the following formula:
\begin{gather*}\text{Error}_2(\tau,h):=  \sqrt{h^2\sum_{ij}|u^M_{ij}(\tau,h)-u^{M}_{2i,2j}(\tau,h/2)|^2},\\  
\text{Order}_2(\tau,h):= \log_2(\text{Error}_2(\tau,h)/\text{Error}_2(\tau,h/2)),
\end{gather*}
where $\tau$ is a small fixed positive number.

\begin{example}\label{exp:sine-Gordon}
In this example, we consider the fractional Laplacian sine-Gordon equation, i.e., $g(u)=-\sin(u)$.
The domain $\Omega$ is set to be $(-10,10)\times(-10,10)$. And the initial conditions are chosen as
\[\varphi_1(x,y)=0,\quad \varphi_2(x,y)=\sech(\sqrt{x^2+y^2}).\]

To assess the time convergence order of the S-ADI scheme, we set the space step size to $h=1/40$, and apply the S-ADI scheme and non-ADI scheme with time step sizes $\tau=1/(10\times 2^{k-1})$ for $k=1,2,3,4$ to obtain the numerical solutions of the fractional Laplacian sine-Gordon equation with fractional orders $\alpha=1.1,1.5,1.9$ at $t=5$.
Similarly, to test the space convergence order of the S-ADI scheme, we set the time step size to $\tau=1/100$, and apply the S-ADI scheme and non-ADI scheme with space step sizes $h=1/2^{k-1}$ for $k=1,2,3,4$ to obtain the numerical solutions of the fractional Laplacian sine-Gordon equation with fractional orders $\alpha=1.1,1.5,1.9$ at $t=5$.
The numerical results in Tables \ref{table:sine-Gordon time convergence} and \ref{table:sine-Gordon space convergence} indicate that the S-ADI scheme achieves second-order convergence and has the similar accuracy as the non-ADI scheme. Furthermore, due to the ADI technique, the S-ADI scheme is much more computationally efficient than the non-ADI scheme.

To provide insight into the numerical solution, we plot $\sin(u^n_{ij})$ at $t=1.25,2.50,3.75,5.00$ obtained by the S-ADI scheme with step sizes $\tau=1/100$ and $h=1/40$ for the fractional Laplacian sine-Gordon equation with fractional orders $\alpha=1.1,1.5,1.9$ in Figure \ref{fig:sine-Gordon numerical solution}. We observe that the movement of the ring solutions is different for different fractional orders.

\begin{table}[!htbp]
    \caption{Errors, time convergence orders and CPU times of S-ADI scheme and non-ADI scheme with $h=1/40$  for Example \ref{exp:sine-Gordon}.}\label{table:sine-Gordon time convergence}
    \centering
    \begin{tabular}{llllllll}
    \toprule
        ~&~&\multicolumn{3}{c}{S-ADI}&\multicolumn{3}{c}{non-ADI}\\ \cmidrule(r){3-5} \cmidrule(r){6-8}
        $\alpha$ & $\tau$ & $\text{Error}_1$ & $\text{Order}_1$ & CPU & $\text{Error}_1$ & $\text{Order}_1$ & CPU \\ \hline
        1.1 &    1/10  & 9.5565E-03 & ~ & 30.2224 & 7.7623E-03 & ~ & 140.3881 \\
        ~ &    1/20  & 2.3936E-03 & 1.9973 & 63.1951 & 1.9523E-03 & 1.9913 & 251.1600 \\
        ~ &    1/40  & 5.9869E-04 & 1.9993 & 120.0303 & 4.8883E-04 & 1.9978 & 451.5331 \\
        ~ &    1/80  & 1.4969E-04 & 1.9998 & 223.4707 & 1.2225E-04 & 1.9994 & 944.6361 \\ \hline
        1.5 &    1/10  & 8.9900E-03 & ~ & 31.9446 & 8.2681E-03 & ~ & 169.4518 \\
        ~ &    1/20  & 2.2583E-03 & 1.9931 & 58.6967 & 2.0890E-03 & 1.9847 & 300.7541 \\
        ~ &    1/40  & 5.6526E-04 & 1.9983 & 113.4957 & 5.2366E-04 & 1.9961 & 511.2563 \\
        ~ &    1/80  & 1.4136E-04 & 1.9995 & 223.1865 & 1.3100E-04 & 1.9990 & 909.0794 \\ \hline
        1.9 &    1/10  & 9.3891E-03 & ~ & 30.5955 & 9.5583E-03 & ~ & 104.2363 \\
        ~ &    1/20  & 2.3670E-03 & 1.9879 & 57.6734 & 2.4320E-03 & 1.9746 & 197.2143 \\
        ~ &    1/40  & 5.9299E-04 & 1.9970 & 113.5512 & 6.1075E-04 & 1.9935 & 351.0598 \\
        ~ &    1/80  & 1.4834E-04 & 1.9991 & 229.9328 & 1.5287E-04 & 1.9983 & 698.5949 \\   \bottomrule
    \end{tabular}
\end{table}

\begin{table}[!htbp]
    \caption{Errors, space convergence orders and CPU times of S-ADI scheme and non-ADI scheme with $\tau=1/100$  for Example \ref{exp:sine-Gordon}.}
    \centering
    \begin{tabular}{llllllll}
    \toprule
        ~&~&\multicolumn{3}{c}{S-ADI}&\multicolumn{3}{c}{non-ADI}\\ \cmidrule(r){3-5} \cmidrule(r){6-8}
        $\alpha$ & $h$ & $\text{Error}_2$ & $\text{Order}_2$ & CPU & $\text{Error}_2$ & $\text{Order}_2$ & CPU \\ \hline
        1.1 & 1         & 6.3150E-02 & ~ & 3.3870 & 6.3146E-02 & ~ & 3.9880 \\
        ~ &    1/2   & 1.3930E-02 & 2.1806 & 3.5444 & 1.3929E-02 & 2.1806 & 5.1784 \\
        ~ &    1/4   & 3.5042E-03 & 1.9910 & 4.6666 & 3.5040E-03 & 1.9910 & 10.3250 \\
        ~ &    1/8   & 9.0305E-04 & 1.9562 & 9.5196 & 9.0300E-04 & 1.9562 & 31.3643 \\ \hline
        1.5 & 1         & 8.1276E-02 & ~ & 3.2160 & 8.1268E-02 & ~ & 3.9510 \\
        ~ &    1/2   & 2.0729E-02 & 1.9712 & 3.4648 & 2.0726E-02 & 1.9712 & 4.7865 \\
        ~ &    1/4   & 5.1991E-03 & 1.9953 & 4.6554 & 5.1985E-03 & 1.9953 & 10.2934 \\
        ~ &    1/8   & 1.3009E-03 & 1.9988 & 9.4895 & 1.3007E-03 & 1.9988 & 31.6671 \\ \hline
        1.9 & 1         & 1.0716E-01 & ~ & 3.2330 & 1.0715E-01 & ~ & 3.9734 \\
        ~ &    1/2   & 2.7583E-02 & 1.9579 & 3.4256 & 2.7578E-02 & 1.9580 & 4.7929 \\
        ~ &    1/4   & 6.9475E-03 & 1.9892 & 4.6446 & 6.9462E-03 & 1.9892 & 9.0031 \\
        ~ &    1/8   & 1.7379E-03 & 1.9991 & 9.5679 & 1.7376E-03 & 1.9991 & 25.2979 \\  \bottomrule
    \end{tabular}\label{table:sine-Gordon space convergence}
\end{table}

\begin{figure}
    \centering

    \subfigure[$\alpha=1.1$, $t=1.25$]{
    \begin{minipage}[c]{0.25\textwidth}
    \centering
    \includegraphics[width=1\textwidth]{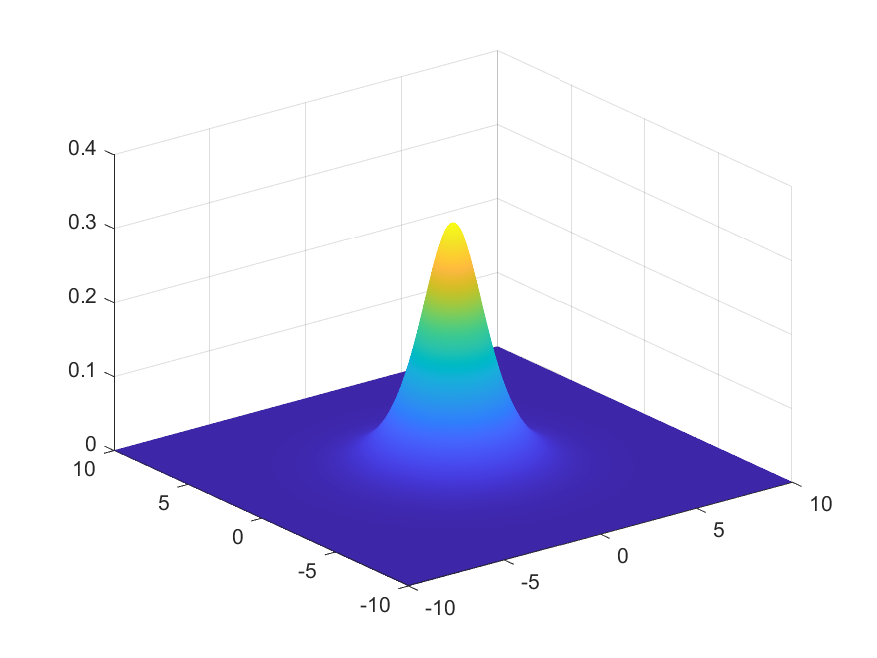}
    \end{minipage}
}
\!\!\!\!\!\!\!\!\!\!\!\!\!\!
   \subfigure[$\alpha=1.1$, $t=2.50$]{
    \begin{minipage}[c]{0.25\textwidth}
    \centering
    \includegraphics[width=1\textwidth]{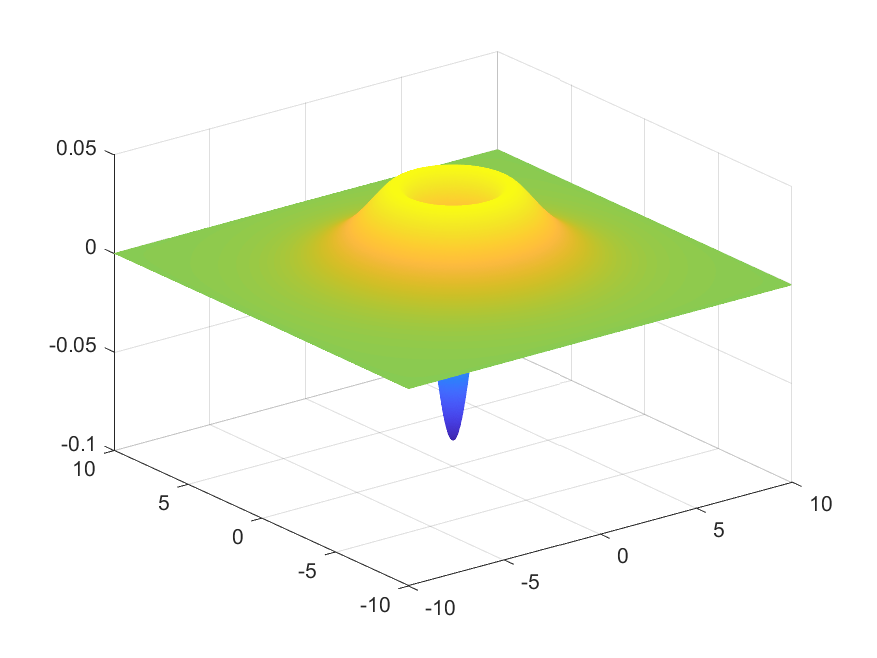}
    \end{minipage}
}
\!\!\!\!\!\!\!\!\!\!\!\!\!\!
    \subfigure[$\alpha=1.1$, $t=3.75$]{
    \begin{minipage}[c]{0.25\textwidth}
    \centering
    \includegraphics[width=1\textwidth]{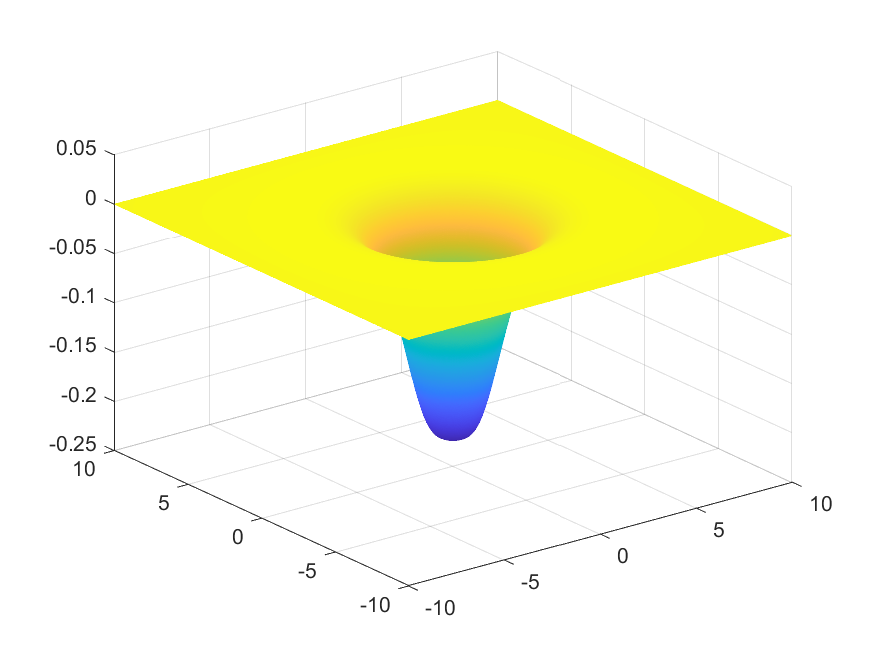}
    \end{minipage}
}
\!\!\!\!\!\!\!\!\!\!\!\!\!\!
   \subfigure[$\alpha=1.1$, $t=5.00$]{
    \begin{minipage}[c]{0.25\textwidth}
    \centering
    \includegraphics[width=1\textwidth]{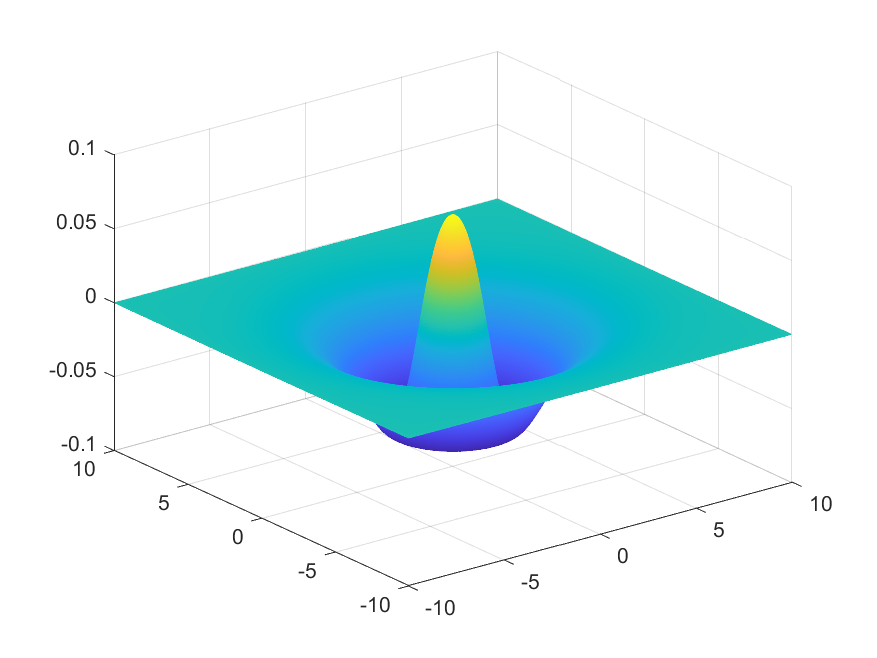}
    \end{minipage}
}

    \subfigure[$\alpha=1.5$, $t=1.25$]{
    \begin{minipage}[c]{0.25\textwidth}
    \centering
    \includegraphics[width=1\textwidth]{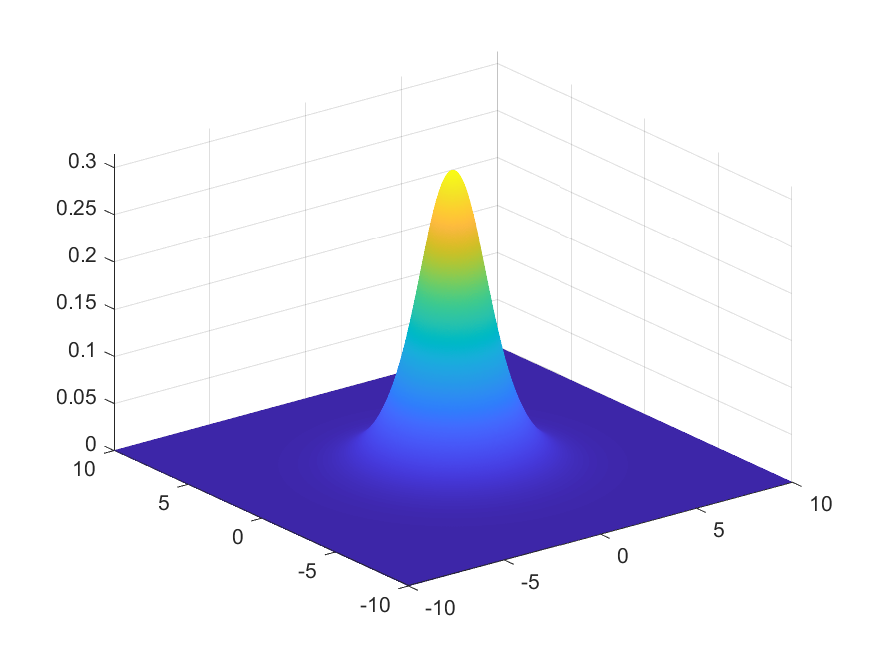}
    \end{minipage}
}
\!\!\!\!\!\!\!\!\!\!\!\!\!\!
   \subfigure[$\alpha=1.5$, $t=2.50$]{
    \begin{minipage}[c]{0.25\textwidth}
    \centering
    \includegraphics[width=1\textwidth]{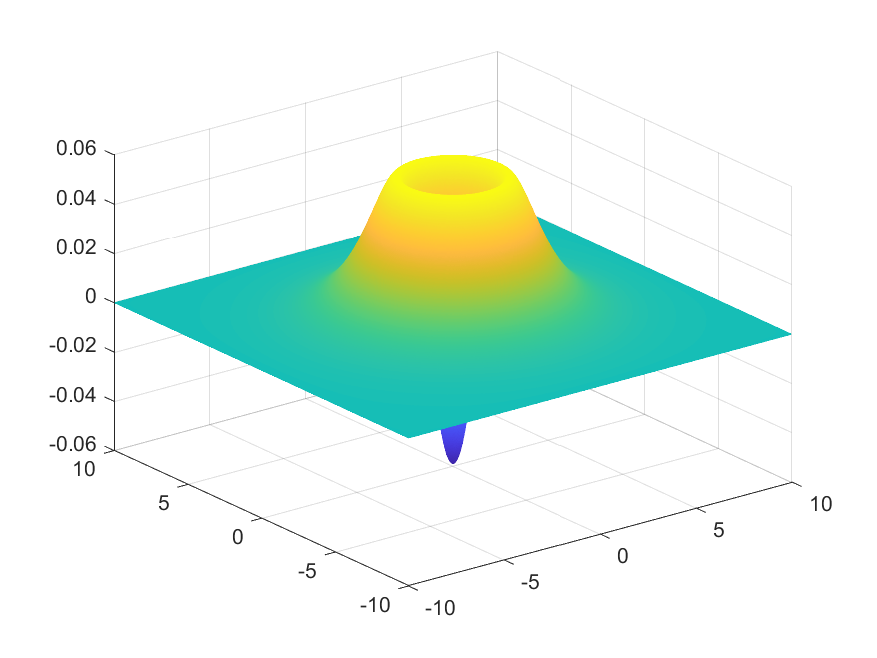}
    \end{minipage}
}
\!\!\!\!\!\!\!\!\!\!\!\!\!\!
    \subfigure[$\alpha=1.5$, $t=3.75$]{
    \begin{minipage}[c]{0.25\textwidth}
    \centering
    \includegraphics[width=1\textwidth]{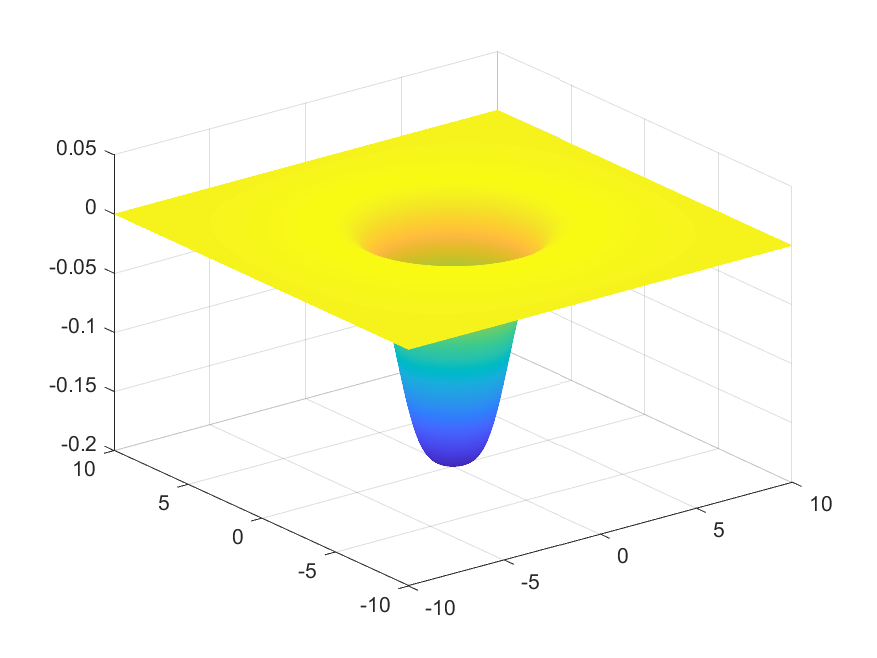}
    \end{minipage}
}
\!\!\!\!\!\!\!\!\!\!\!\!\!\!
   \subfigure[$\alpha=1.5$, $t=5.00$]{
    \begin{minipage}[c]{0.25\textwidth}
    \centering
    \includegraphics[width=1\textwidth]{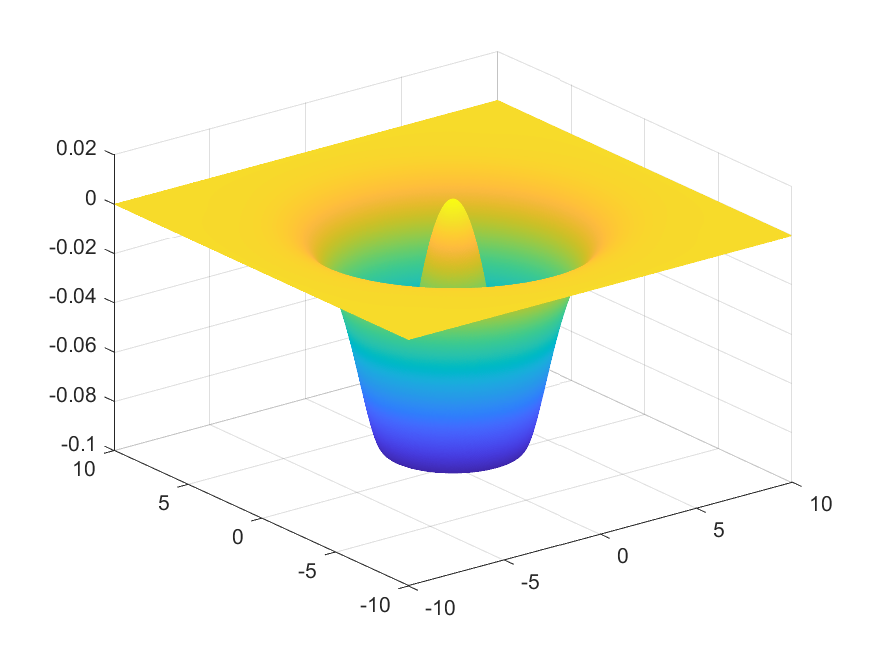}
    \end{minipage}
}

    \subfigure[$\alpha=1.9$, $t=1.25$]{
    \begin{minipage}[c]{0.25\textwidth}
    \centering
    \includegraphics[width=1\textwidth]{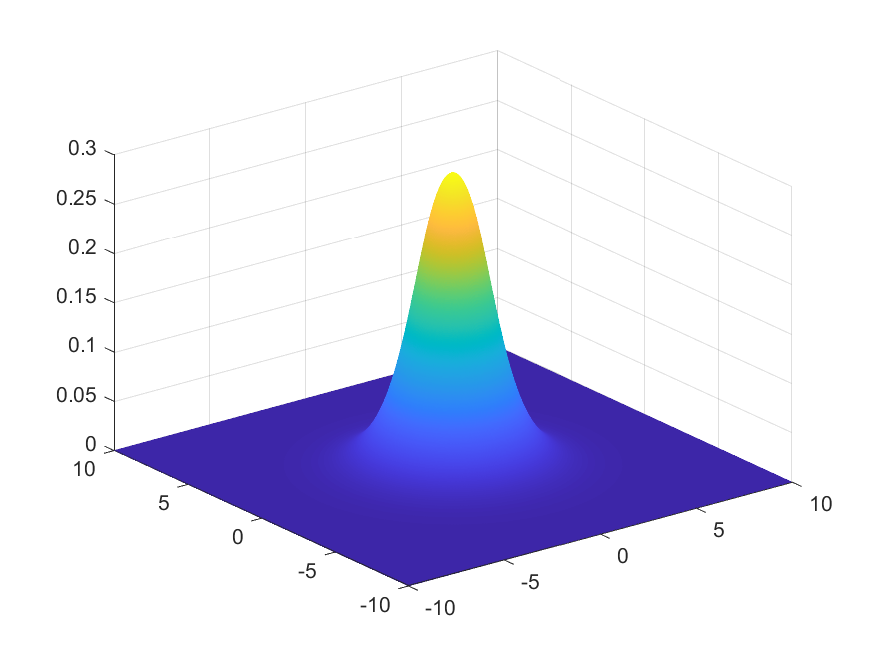}
    \end{minipage}
}
\!\!\!\!\!\!\!\!\!\!\!\!\!\!
   \subfigure[$\alpha=1.9$, $t=2.50$]{
    \begin{minipage}[c]{0.25\textwidth}
    \centering
    \includegraphics[width=1\textwidth]{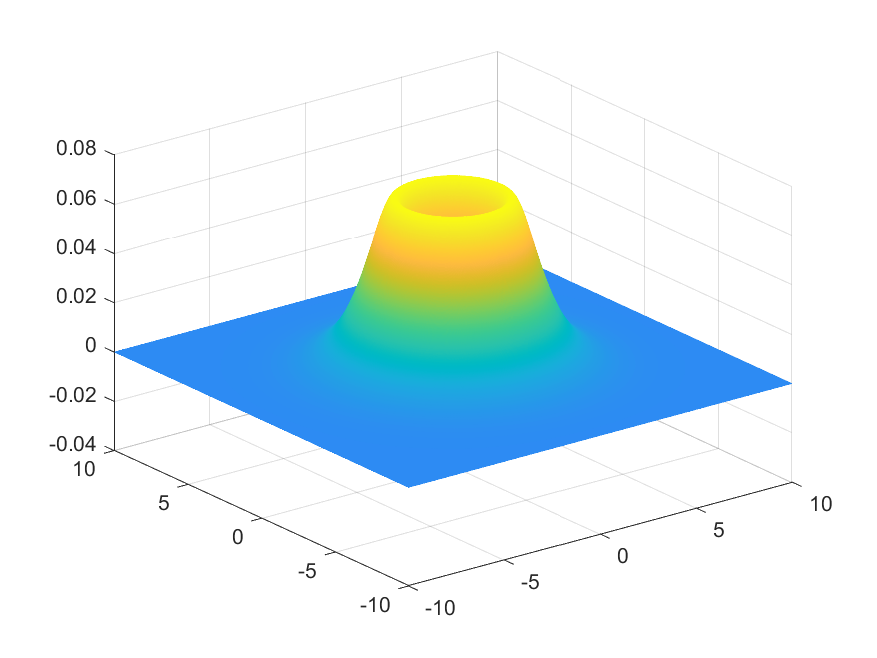}
    \end{minipage}
}
\!\!\!\!\!\!\!\!\!\!\!\!\!\!
    \subfigure[$\alpha=1.9$, $t=3.75$]{
    \begin{minipage}[c]{0.25\textwidth}
    \centering
    \includegraphics[width=1\textwidth]{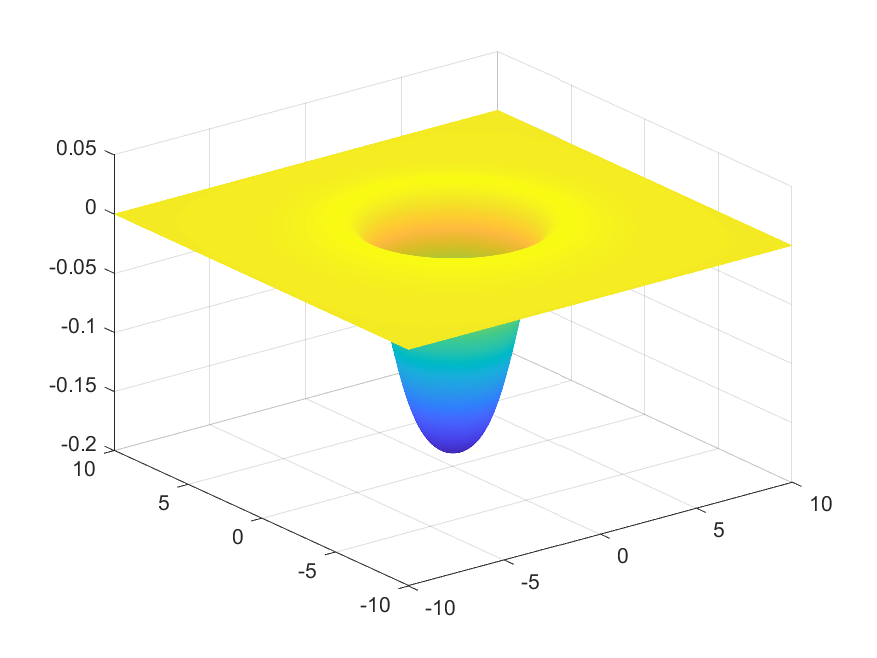}
    \end{minipage}
}
\!\!\!\!\!\!\!\!\!\!\!\!\!\!
   \subfigure[$\alpha=1.9$, $t=5.00$]{
    \begin{minipage}[c]{0.25\textwidth}
    \centering
    \includegraphics[width=1\textwidth]{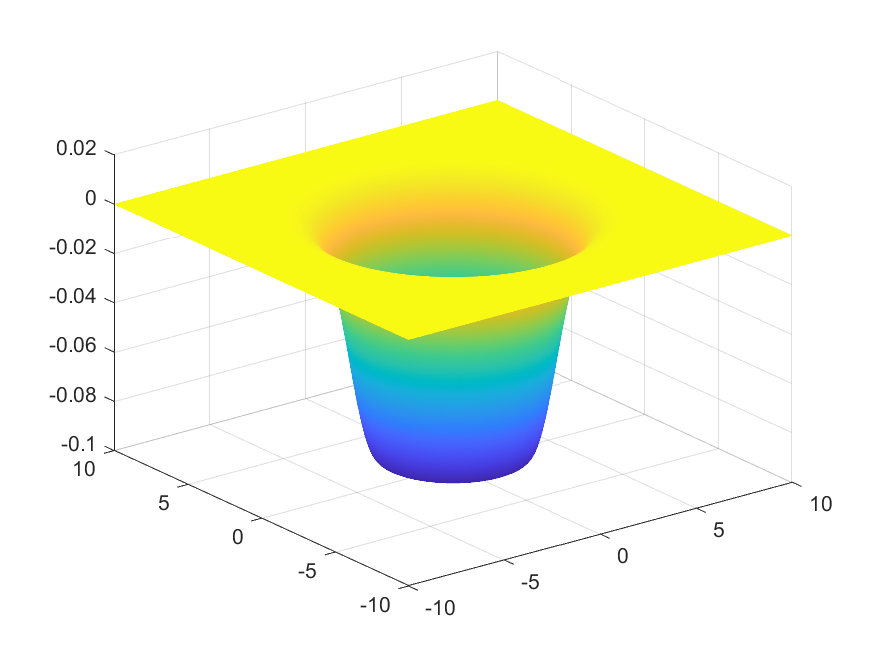}
    \end{minipage}
}

    \caption{The surfaces of $\sin(u^n_{ij}/2)$ by S-ADI scheme with $\tau=1/100$ and $h=1/40$ for Example \ref{exp:sine-Gordon}.}
    \label{fig:sine-Gordon numerical solution}
\end{figure}

\end{example}

\begin{example}\label{exp:Klein-Gordon}
In this example, we consider the fractional Laplacian Klein-Gordon equation, i.e., $g(u)=-u^3$.
The domain $\Omega$ is set to be $(-10,10)\times(-10,10)$. And the initial conditions are chosen as \cite{Hu_Cai_2021_2}:
\[\varphi_1(x,y)=\sech(\cosh(x^2+y^2)),\quad \varphi_2(x,y)=0.\]

We apply both the S-ADI scheme and non-ADI scheme to compute the solution of the fractional Klein-Gordon equation for different fractional orders $\alpha=1.1,1.5,1.9$ at $t=8$. In order to demonstrate the time convergence order, we fix the space step size to be $h=1/50$ and vary the time step size as $\tau=4/(25\times 2^k)(k=0,1,2,3)$. 
To manifest the space convergence order, we fix the time step size to be $\tau=1/125$ and vary the space step size as $h=2/(5\times 2^k)(k=0,1,2,3)$. The computational results are presented in Tables \ref{table:Klein-Gordon time convergence} and \ref{table:Klein-Gordon space convergence}.
We observe that the S-ADI scheme achieves a second-order convergence order.
 Moreover, the computational cost of the S-ADI scheme is significantly lower than the non-ADI scheme.

To gain further insight into the numerical solutions, we plot the surfaces of $u^n_{ij}$ at $t=2,4,6,8$ using the S-ADI scheme with step size $\tau=1/100$ and $h=1/40$ for the fractional Klein-Gordon equation with different fractional orders $\alpha=1.1,1.5,1.9$. These plots are shown in Figure \ref{fig:Klein-Gordon numerical solution}. It is clear from the plots that the spread speeds of the initial soliton for different fractional orders are distinct.

\begin{table}[!htbp]
    \caption{Errors, time convergence orders and CPU times of S-ADI scheme and non-ADI scheme with $h=1/50$  for Example \ref{exp:Klein-Gordon}.}
    \centering
    \begin{tabular}{llllllll}
    \toprule
        ~&~&\multicolumn{3}{c}{S-ADI}&\multicolumn{3}{c}{non-ADI}\\ \cmidrule(r){3-5} \cmidrule(r){6-8}
        $\alpha$ & $\tau$ & $\text{Error}_1$ & $\text{Order}_1$ & CPU & $\text{Error}_1$ & $\text{Order}_1$ & CPU \\ \hline
        1.1 &    4/25  & 2.2313E-01 & ~ & 57.8186 & 1.1807E-01 & ~ & 305.3089 \\
        ~ &    2/25  & 5.7897E-02 & 1.9463 & 112.9196 & 2.9612E-02 & 1.9954 & 498.1011 \\
        ~ &    1/25  & 1.4543E-02 & 1.9932 & 216.5971 & 1.5700E-02 & 0.9154 & 1025.5165 \\
        ~ &    1/50  & 3.6394E-03 & 1.9985 & 428.0685 & 9.3372E-03 & 0.7497 & 1692.2861 \\ \hline
        1.5 &    4/25  & 3.5383E-01 & ~ & 57.3419 & 2.6350E-01 & ~ & 318.2574 \\
        ~ &    2/25  & 1.0344E-01 & 1.7742 & 109.3765 & 6.7303E-02 & 1.9691 & 529.6633 \\
        ~ &    1/25  & 2.6717E-02 & 1.9530 & 217.2319 & 1.8568E-02 & 1.8579 & 977.5806 \\
        ~ &    1/50  & 6.7184E-03 & 1.9916 & 430.4121 & 9.1443E-03 & 1.0219 & 1672.2912 \\ \hline
        1.9 &    4/25  & 5.2869E-01 & ~ & 57.4014 & 4.8025E-01 & ~ & 255.7158 \\
        ~ &    2/25  & 2.0867E-01 & 1.3412 & 110.0169 & 1.8102E-01 & 1.4076 & 491.6168 \\
        ~ &    1/25  & 6.3176E-02 & 1.7238 & 214.5857 & 5.1215E-02 & 1.8215 & 853.5335 \\
        ~ &    1/50  & 1.6467E-02 & 1.9398 & 432.4007 & 1.3739E-02 & 1.8983 & 1652.7898 \\    \bottomrule
    \end{tabular}\label{table:Klein-Gordon time convergence}
\end{table}

\begin{table}[!htbp]
    \caption{Errors, space convergence orders and CPU times of S-ADI scheme and non-ADI scheme with $\tau=1/125$  for Example \ref{exp:Klein-Gordon}.}
    \centering
    \begin{tabular}{llllllll}
    \toprule
        ~&~&\multicolumn{3}{c}{S-ADI}&\multicolumn{3}{c}{non-ADI}\\ \cmidrule(r){3-5} \cmidrule(r){6-8}
        $\alpha$ & $h$ & $\text{Error}_2$ & $\text{Order}_2$ & CPU & $\text{Error}_2$ & $\text{Order}_2$ & CPU \\ \hline
        1.1 &    2/5   & 2.6371E-01 & ~ & 6.9554 & 2.6388E-01 & ~ & 15.0042 \\
        ~ &    1/5   & 7.9964E-02 & 1.7215 & 8.3911 & 8.0038E-02 & 1.7212 & 30.3791 \\
        ~ &    1/10  & 2.0179E-02 & 1.9865 & 33.6068 & 2.0176E-02 & 1.9880 & 151.0825 \\
        ~ &    1/20  & 4.9825E-03 & 2.0179 & 115.2789 & 4.9817E-03 & 2.0180 & 464.6734 \\ \hline
        1.5 &    2/5   & 3.9151E-01 & ~ & 5.9673 & 3.9167E-01 & ~ & 13.6257 \\
        ~ &    1/5   & 1.3400E-01 & 1.5468 & 10.9536 & 1.3408E-01 & 1.5465 & 27.9086 \\
        ~ &    1/10  & 3.7027E-02 & 1.8556 & 32.6669 & 3.7051E-02 & 1.8555 & 77.0358 \\
        ~ &    1/20  & 9.3993E-03 & 1.9779 & 116.0624 & 9.4057E-03 & 1.9779 & 284.3546 \\ \hline
        1.9 &    2/5   & 5.1604E-01 & ~ & 6.2365 & 5.1626E-01 & ~ & 14.2160 \\
        ~ &    1/5   & 2.1356E-01 & 1.2729 & 10.9624 & 2.1368E-01 & 1.2726 & 29.1248 \\
        ~ &    1/10  & 6.6559E-02 & 1.6819 & 31.3464 & 6.6610E-02 & 1.6817 & 125.6665 \\
        ~ &    1/20  & 1.7504E-02 & 1.9269 & 117.8960 & 1.7519E-02 & 1.9268 & 415.6226 \\  \bottomrule
    \end{tabular}\label{table:Klein-Gordon space convergence}
\end{table}

\begin{figure}
    \centering

    \subfigure[$\alpha=1.1$, $t=2$]{
    \begin{minipage}[c]{0.25\textwidth}
    \centering
    \includegraphics[width=1\textwidth]{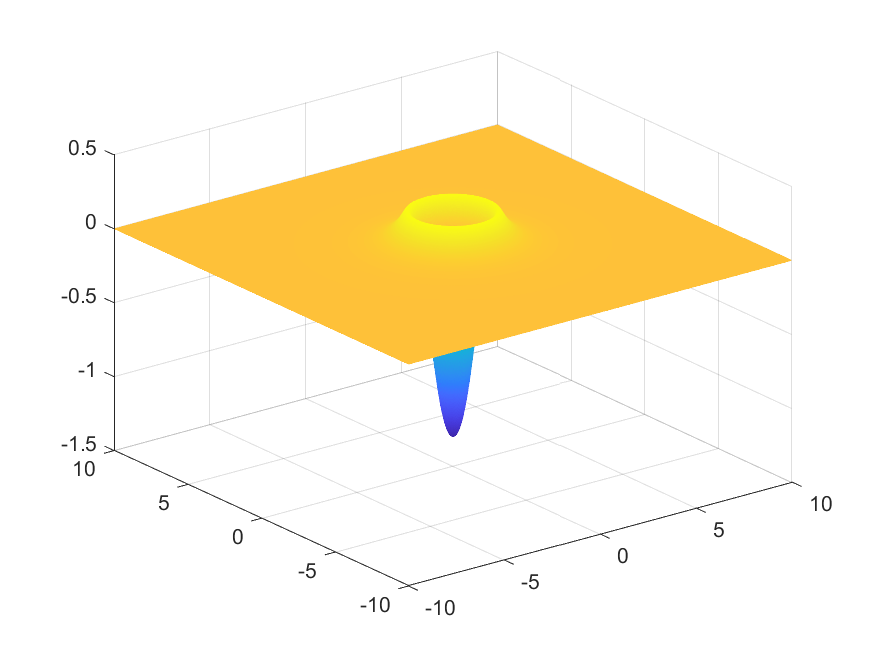}
    \end{minipage}
}
\!\!\!\!\!\!\!\!\!\!\!\!\!\!
   \subfigure[$\alpha=1.1$, $t=4$]{
    \begin{minipage}[c]{0.25\textwidth}
    \centering
    \includegraphics[width=1\textwidth]{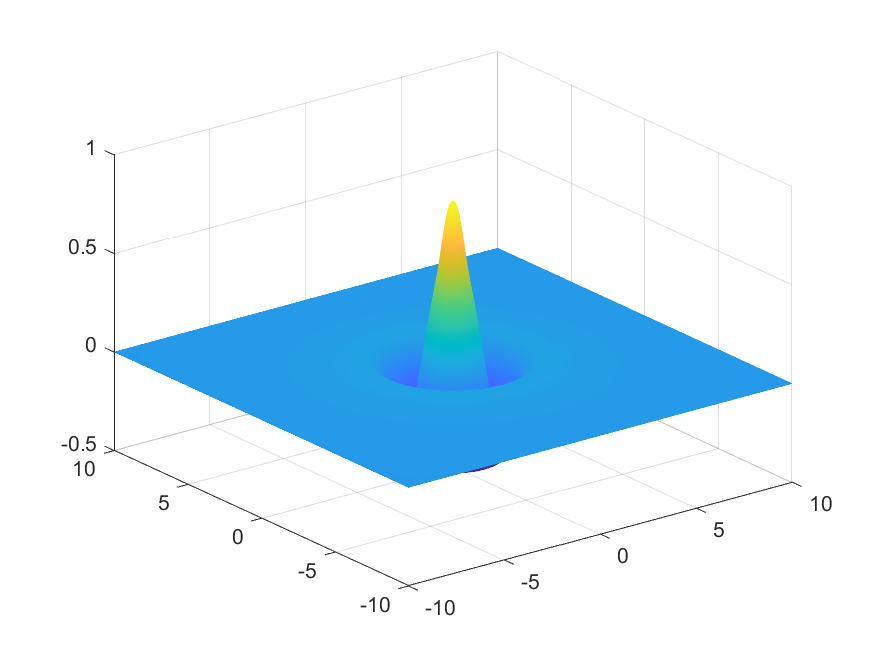}
    \end{minipage}
}
\!\!\!\!\!\!\!\!\!\!\!\!\!\!
    \subfigure[$\alpha=1.1$, $t=6$]{
    \begin{minipage}[c]{0.25\textwidth}
    \centering
    \includegraphics[width=1\textwidth]{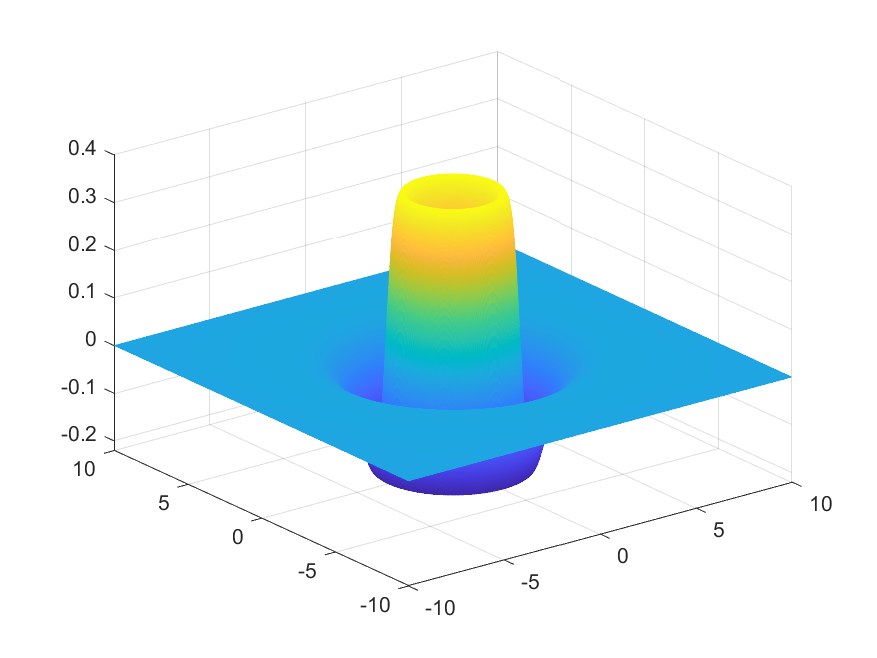}
    \end{minipage}
}
\!\!\!\!\!\!\!\!\!\!\!\!\!\!
   \subfigure[$\alpha=1.1$, $t=8$]{
    \begin{minipage}[c]{0.25\textwidth}
    \centering
    \includegraphics[width=1\textwidth]{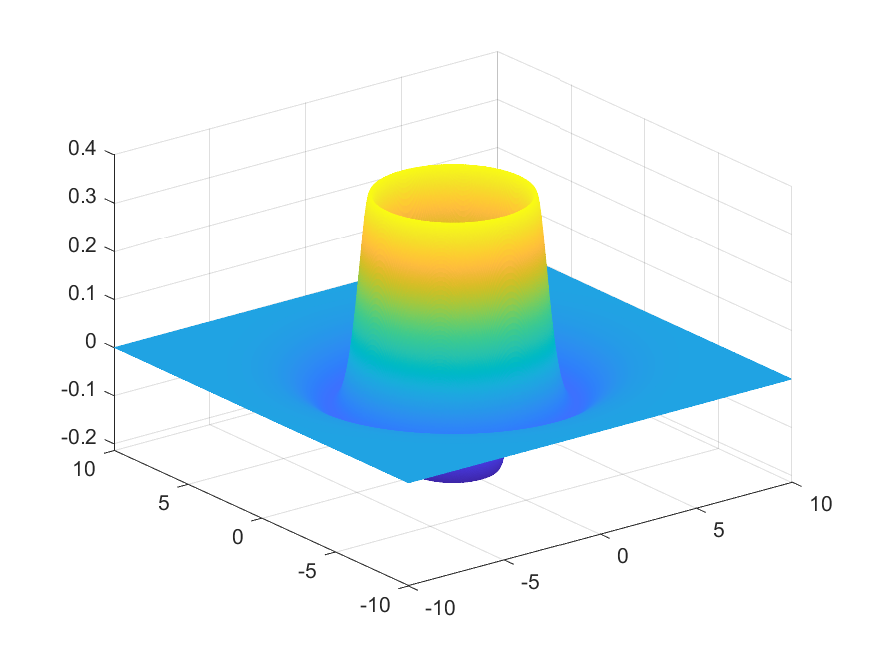}
    \end{minipage}
}

    \subfigure[$\alpha=1.5$, $t=2$]{
    \begin{minipage}[c]{0.25\textwidth}
    \centering
    \includegraphics[width=1\textwidth]{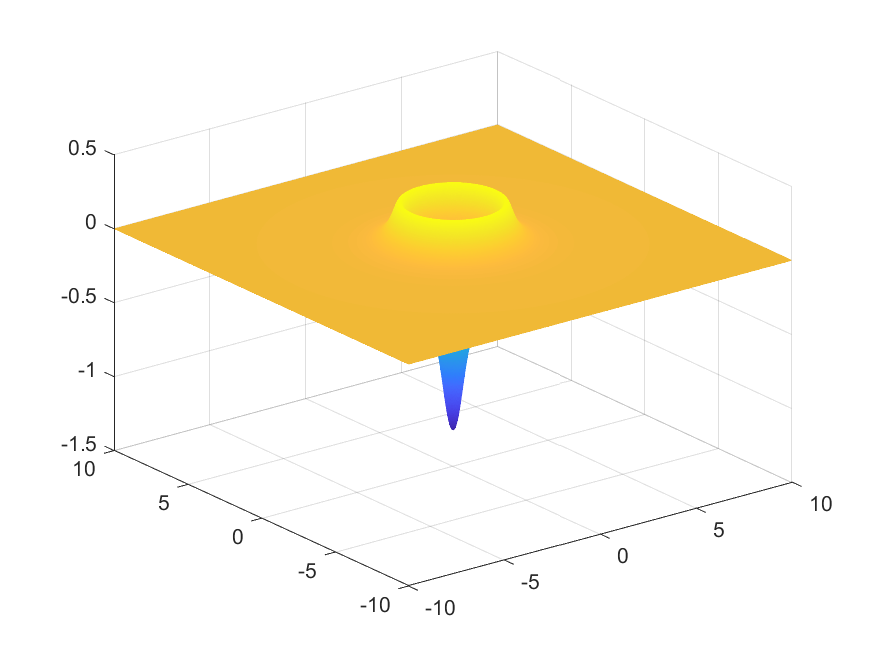}
    \end{minipage}
}
\!\!\!\!\!\!\!\!\!\!\!\!\!\!
   \subfigure[$\alpha=1.5$, $t=4$]{
    \begin{minipage}[c]{0.25\textwidth}
    \centering
    \includegraphics[width=1\textwidth]{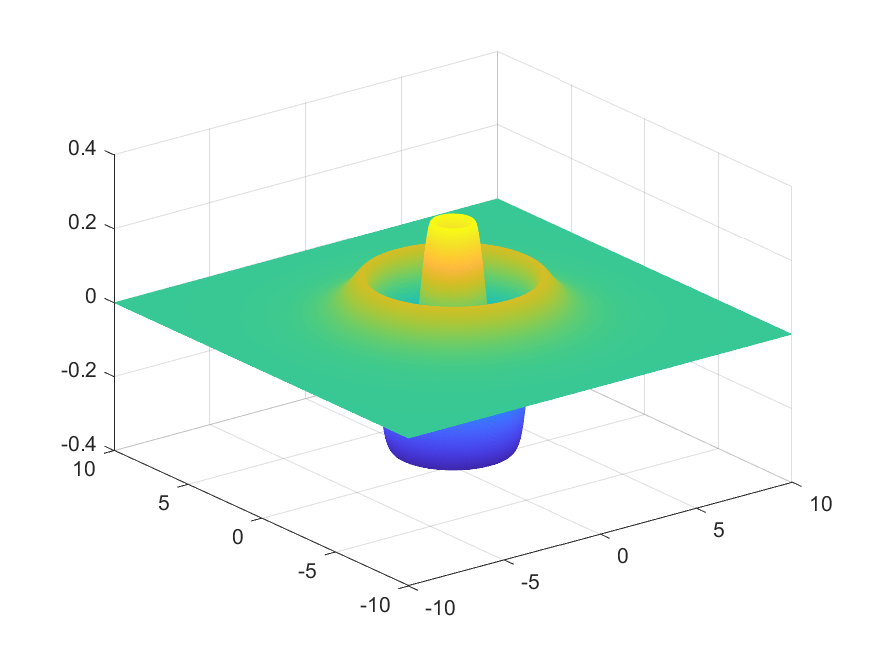}
    \end{minipage}
}
\!\!\!\!\!\!\!\!\!\!\!\!\!\!
    \subfigure[$\alpha=1.5$, $t=6$]{
    \begin{minipage}[c]{0.25\textwidth}
    \centering
    \includegraphics[width=1\textwidth]{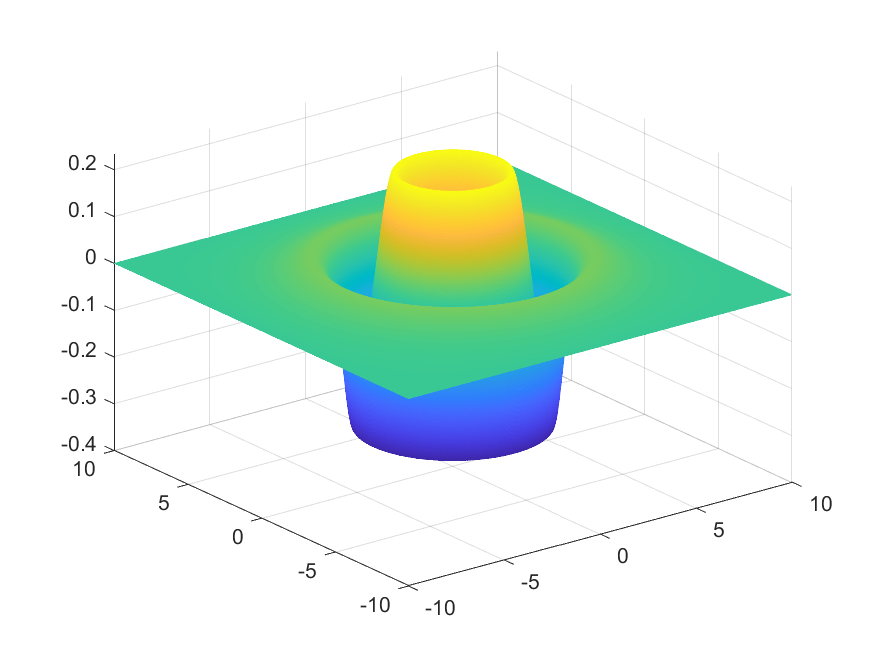}
    \end{minipage}
}
\!\!\!\!\!\!\!\!\!\!\!\!\!\!
   \subfigure[$\alpha=1.5$, $t=8$]{
    \begin{minipage}[c]{0.25\textwidth}
    \centering
    \includegraphics[width=1\textwidth]{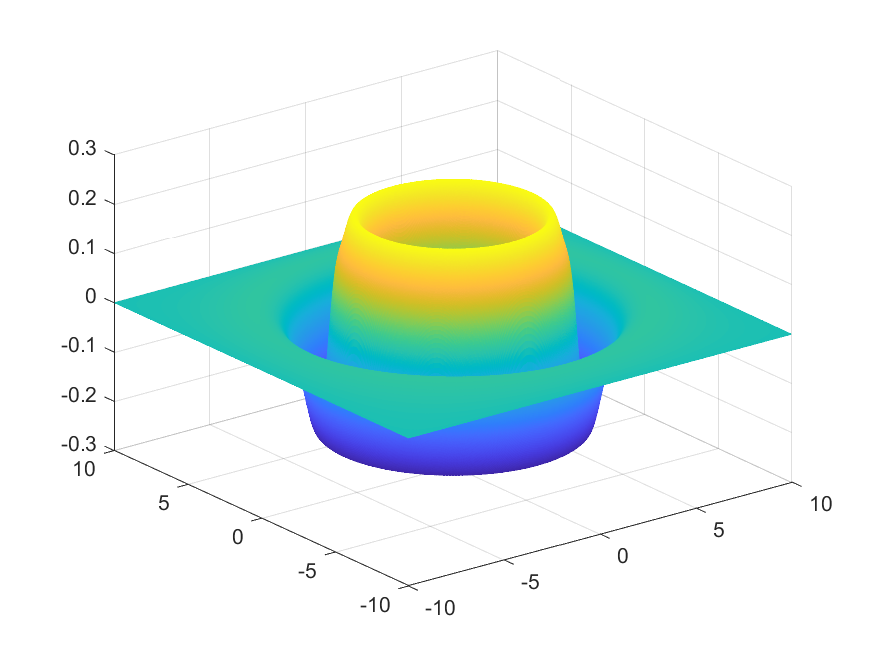}
    \end{minipage}
}

    \subfigure[$\alpha=1.9$, $t=2$]{
    \begin{minipage}[c]{0.25\textwidth}
    \centering
    \includegraphics[width=1\textwidth]{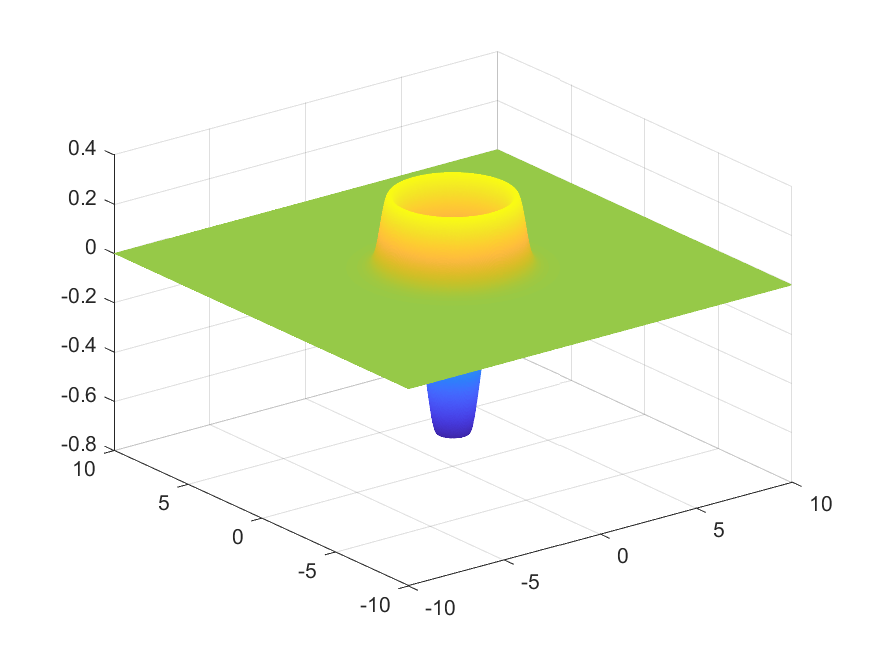}
    \end{minipage}
}
\!\!\!\!\!\!\!\!\!\!\!\!\!\!
   \subfigure[$\alpha=1.9$, $t=4$]{
    \begin{minipage}[c]{0.25\textwidth}
    \centering
    \includegraphics[width=1\textwidth]{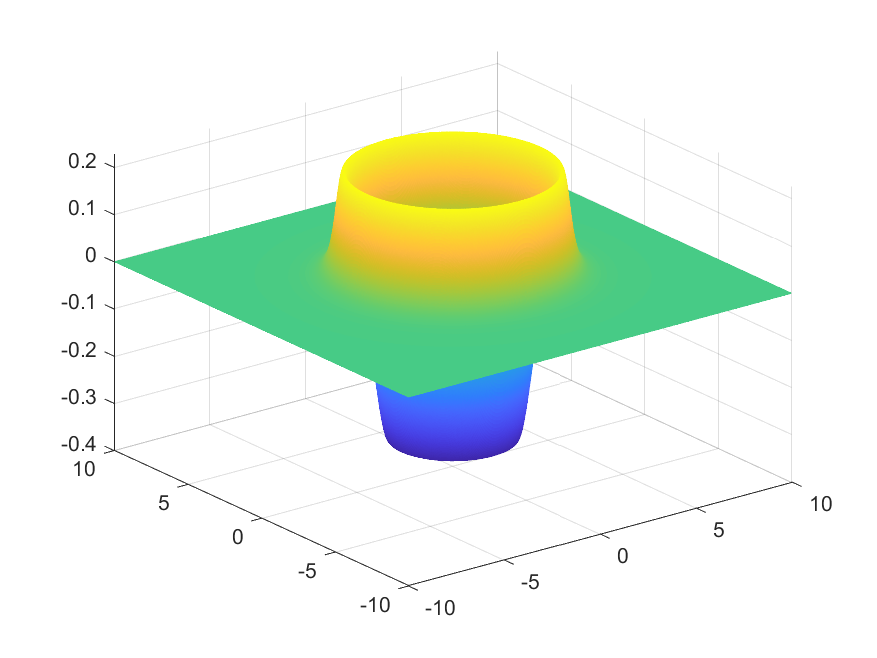}
    \end{minipage}
}
\!\!\!\!\!\!\!\!\!\!\!\!\!\!
    \subfigure[$\alpha=1.9$, $t=6$]{
    \begin{minipage}[c]{0.25\textwidth}
    \centering
    \includegraphics[width=1\textwidth]{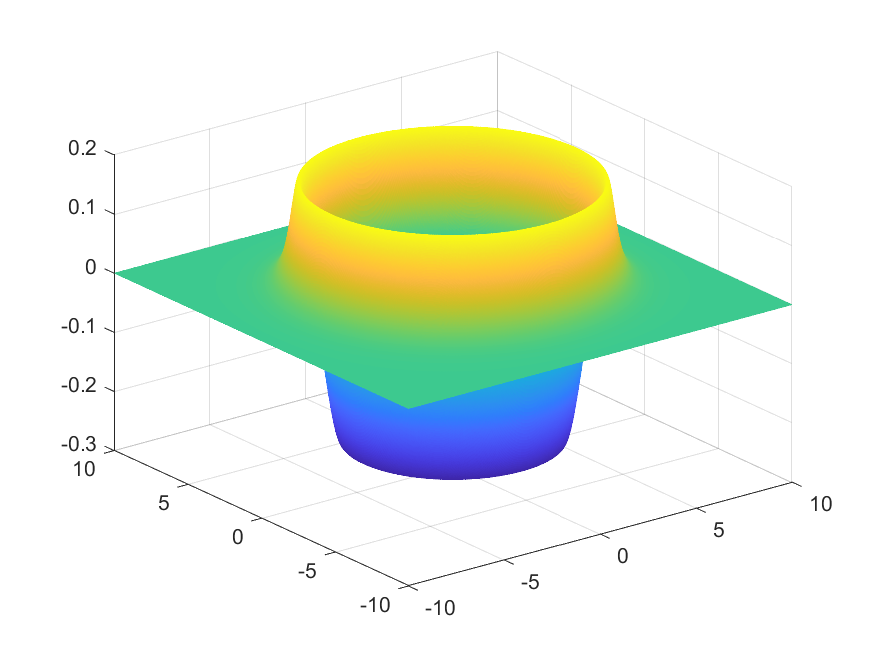}
    \end{minipage}
}
\!\!\!\!\!\!\!\!\!\!\!\!\!\!
   \subfigure[$\alpha=1.9$, $t=8$]{
    \begin{minipage}[c]{0.25\textwidth}
    \centering
    \includegraphics[width=1\textwidth]{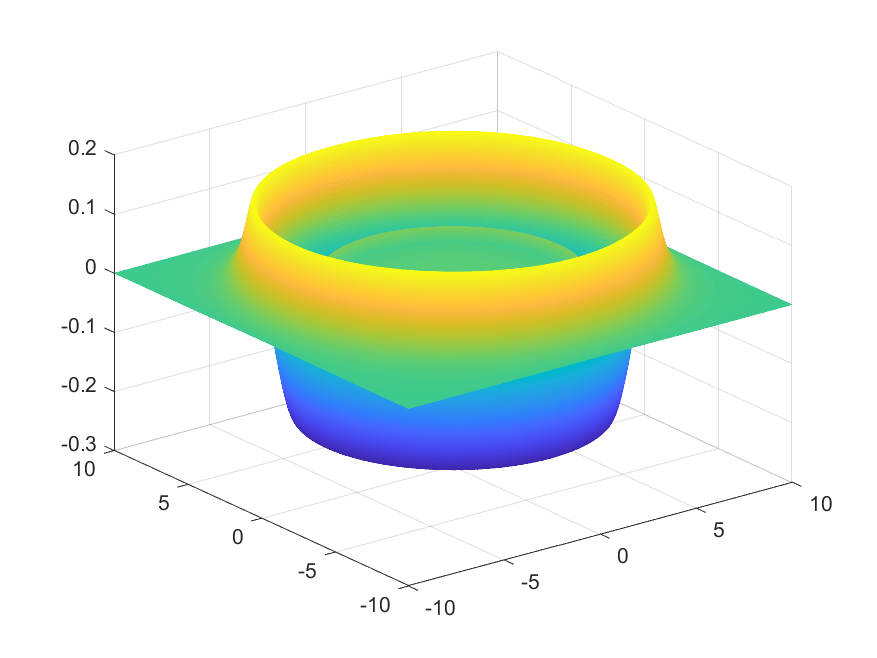}
    \end{minipage}
}

    \caption{The numerical solutions of S-ADI scheme with $\tau=1/100$ and $h=1/40$ for Example \ref{exp:Klein-Gordon}.}
    \label{fig:Klein-Gordon numerical solution}
\end{figure}

\end{example}

\section{Conclusion}\label{section:conclusion}

The nonlocal nature of the fractional Laplacian poses a challenge for constructing ADI schemes to solve partial differential equations involving this operator. In this work, we propose a second-order accuracy S-ADI scheme for the fractional Laplacian wave equation by using the operator splitting technique and splitting out the discrete Riesz fractional derivatives. Thanks to the ADI technique, the resulted Toeplitz linear systems are efficiently solved by GS formula. We rigorously prove the unconditional stability and convergence of the proposed scheme. To validate the performance of the S-ADI scheme, we apply it to solve the fractional Laplacian sine-Gordon equation and Klein-Gordon equation, and compare it with the non-ADI scheme. The numerical results demonstrate the efficiency and accuracy of the S-ADI scheme.

Although the S-ADI scheme proposed in this work is effective for solving the fractional Laplacian wave equation, it has some limitations. Specifically, the simplest explicit scheme used for the nonlinear term means that this scheme cannot conserve certain properties, such as energy invariance, of certain original equations. To address this limitation, future work will explore the use of the scalar auxiliary variable method, as described in \cite{Shen_Xu_2018,Liu_Li_2020}.
This paper focused solely on the fractional Laplacian wave equation, but the S-ADI scheme can also be applied to the fractional Laplacian diffusion equation. Ongoing research is exploring this application.
Besides, it is well-known that the fractional Laplacian is a special nonlocal operator, so it is possible to extend the S-ADI scheme to solve the evolution differential equation with general nonlocal operator.


\section*{Declarations}

\bmhead{Ethical approval} Not Applicable.
\bmhead{Availability of supporting data} The data of codes involved in this paper are available from the corresponding author on reasonable request.
\bmhead{Competing interests} The authors declare no competing interests.
\bmhead{Funding} Science and Technology Development Fund of Macao SAR (Grant No. 0122/2020/A3); MYRG-GRG2023-00085-FST-UMDF from University of Macau.
\bmhead{Authors' contributions} The two authors contributed equally.
\bmhead{Acknowledgments} The second author (corresponding author) Hai-Wei Sun is supported by Science and Technology Development Fund of Macao SAR (Grant No. 0122/2020/A3) and MYRG-GRG2023-00085-FST-UMDF from University of Macau.

\end{document}